\documentclass[a4paper,12pt]{article}

\usepackage{amsmath,amssymb,amsthm,hyperref,textcomp,enumitem}
\usepackage{xcolor}

\newtheorem{thm}{Theorem}
\newtheorem{prop}[thm]{Proposition}
\newtheorem{cor}[thm]{Corollary}
\newtheorem{lemm}[thm]{Lemma}

\theoremstyle{definition}
\newtheorem{assu}[thm]{Assumption}
\newtheorem{example}[thm]{Example}
\newtheorem{rem}[thm]{Remark}
\newtheorem{definition}[thm]{Definition}
\newtheorem*{notation}{Notation}

\makeatletter
\numberwithin{equation}{section}
\allowdisplaybreaks

\makeatother

\def\Delta{\delta}

\def\N{{\mathbb N}} 
\def\Z{{\mathbb Z}}
\def\Q{{\mathbb Q}}
\def\R{{\mathbb R}}
\def\C{{\mathbb C}}

\def\s{{\bf s}}
\def\m{{\bf m}}

\def\x{{\bf x}}

\def\z{{\bf z}}

\def\u{{\bf u}}

\def\Nb{{\bf N}}
\def\P{{\bf P}}

\def\M{{\bf M}}

\newcommand{\zerob}{\boldsymbol{0}}

\newcommand{\alphab}{\boldsymbol{\alpha}}
\newcommand{\betab}{\boldsymbol{\beta}}

\newcommand{\mub}{\boldsymbol{\mu}}

\newcommand{\dc}{\mathrm{dC}}

\newcommand{\dd}{d}
\newcommand{\ii}{i}
\newcommand{\phie}{e}

\def\eps{{\varepsilon}}
\newcommand{\be}{\begin{enumerate}}
\newcommand{\ee}{\end{enumerate}}
\let\ds=\displaystyle

\begin{document}
\title{Values at non-positive integers of partially twisted multiple zeta-functions II}
\author{Driss Essouabri \and Kohji Matsumoto \and Simon Rutard}
\date{\today}
\maketitle
\begin{abstract}
We study the values at non-positive integer points of multi-variable
twisted multiple zeta-functions, whose each factor of the denominator is given by polynomials.    The fully twisted case was
already answered by de Crisenoy.    On the partially twisted case, 
in one of our former article we studied the case when each factor of
the denominator is given by linear forms or power-sum forms.
In the present paper we treat the case of general polynomial
denominators, and obtain explicit forms of the values at non-positive integer points.    Our strategy is to reduce to the theorem of de Crisenoy for the fully twisted case, via the multiple Mellin-Barnes integral formula.     We observe that in some cases the obtained values are transcendental.
\end{abstract}

\paragraph{\textbf{Acknowledgments.}}
The third author was supported by a JSPS Postdoctoral Fellowship (PE24744).

\tableofcontents

\section{Introduction}\label{intro}

Let $\N$, $\N_0$, $\Z$, $\Q$, $\R$, and $\C$ be the set of
all positive integers, all non-negative integers, 
all integers, all rational numbers, all real numbers, and all complex numbers, respectively. Throughout the article, we assume that the sum (respectively product) over the empty set is $0$ (respectively $1$).
It is the aim of the present paper to study the values at non-positive integer points of
the multiple series of the form
\begin{align}\label{1-1}
\zeta_n(\mathbf{s};\mathbf{P};\boldsymbol{\mu}_k)=\sum_{m_1,\ldots,m_n\geq 1}
\frac{\prod_{j=1}^{k}\mu_j^{m_j}}{\left(\prod_{j=1}^{T-1} P_j(m_1,\ldots,m_{n-1})^{s_j}\right)P_T(m_1,\ldots,m_n)^{s_T}},
\end{align}
where $\mathbf{s}=(s_1,\ldots,s_T)\in\mathbb{C}^T$, 
$\mathbf{P}=(P_1,\ldots,P_T)$ with $P_j(X_1,\ldots,X_{n-1})\in \mathbb{R}[X_1,\ldots,X_{n-1}]$
($1\leq j\leq T-1$), $P_T(X_1,\ldots,X_n)\in\mathbb{R}[X_1,\ldots,X_n]$,
and $\boldsymbol{\mu}_k=(\mu_1,\ldots,\mu_{k})\in(\mathbb{T}\setminus\{1\})^k$
($0\leq k\leq n$), 
with $\mathbb{T}=\{z\in\mathbb{C}\mid |z|=1\}$. 
This $\boldsymbol{\mu}_k$ is a kind of twisted factor of \eqref{1-1}.
The series \eqref{1-1} is absolutely convergent if $\Re s_t$ ($1\leq t\leq T$) are sufficiently large (see Proposition \ref{conv_domain} below), and can be continued meromorphically to the whole complex space $\C^T$ under certain assumptions.
Note that, if $n=k=T=1$ and $P_1(m_1)=m_1$, then \eqref{1-1} coincides with the
Lerch zeta-function 
\begin{align}\label{Lerch_def}
\zeta_{\mu_1}(s_1)=\sum_{m_1=1}^{\infty}\mu_1^{m_1}m_1^{-s_1}.
\end{align}
When $n=1$, we understand that the polynomials $P_1,\ldots,P_{T-1}$ correspond to real numbers $p_1,\ldots,p_T \in \mathbb{R} \setminus \{0\}$, and thus
$$\zeta_1(\mathbf{s};\mathbf{P};\boldsymbol{\mu}_1):=p_1^{-s_1} \cdots p_{T-1}^{-s_{T-1}} \sum_{m=1}^{+\infty} \frac{\mu_1^m}{P_T(m)^{s_T}}.$$

If $k=0$, then the numerator on the right-hand side of \eqref{1-1} is an empty product and is regarded as $1$,  so \eqref{1-1} is a non-twisted series.
In particular, if $T=n$ and $P_j(X_1,\ldots,X_j)=X_1+\cdots+X_j$ ($1\leq j\leq n$), then the corresponding
\eqref{1-1} is the classical multiple zeta-functions of Euler-Zagier type
$\zeta_{EZ,n}(\mathbf{s})$.
It is well known that $\zeta_{EZ,n}(\mathbf{s})$ can be continued meromorphically to the whole space, and their
values at non-positive integer points have been extensively studied.  The study of these values was
first cultivated by a pioneering work of Akiyama, Egami and Tanigawa \cite{AET01}.
They noticed that almost all such points are on singularity sets of $\zeta_{EZ,n}(\mathbf{s})$, 
so the ``values'' at those points should be understood as ``limit values''.
They obtained explicit formulas of those limit values in the case of some special cases of
limit processes.
Then, after several subsequent research
(such as Sasaki \cite{Sasa09}), Onozuka \cite{Onoz13}) arrived at an explicit
formula for rather general types of limit processes.
Laurent expansions around those points have further been studied (see Matsumoto, Onozuka
and Wakabayashi \cite{MOW20} and Saha \cite{Saha22}).

The case when $T=n$ and $P_j$ are more general linear forms has also been considered (see
Komori \cite{Komo10}, Essouabri and Matsumoto \cite{EM20}, 
Murahara and Onozuka \cite{MuOn}, Rutard \cite{Rutard}).
For example, in \cite{EM20}, the case
$P_j=\gamma_1X_1+\cdots+\gamma_jX_j+b_j$ ($1\leq j\leq n$, $\gamma_j$, $b_j$ are complex
constants) is studied.    An explicit formula for the limit value of the form
\begin{align}\label{1-3}
\lim_{t\to 0}\zeta_n(-\mathbf{N}+t\boldsymbol{\theta};\mathbf{P};\boldsymbol{\mu}_0)
\end{align}
(where $\mathbf{N}=(N_1,\ldots,N_n)\in \mathbb{N}_0^n$,
$\boldsymbol{\theta}=(\theta_1,\ldots,\theta_n)\in\mathbb{C}^n$)
is obtained, which implies that the value is included in the field generated by the parameters
$\gamma_j$, $b_j$, and $\theta_j$ over the rational number field $\mathbb{Q}$.

Essouabri and Matsumoto \cite{EM21} treated the case when $P_j$ ($1\leq j\leq n$)
are more general polynomials.
They considered the case when $P_j$ are power sums
$\gamma_1 X_1^{d_1}+\cdots+\gamma_j X_j^{d_j}$, or $P_j$ satisfy the $\mathrm{H_0S}$ condition
(see Definition \ref{def_H0S} below) with some additional conditions, under which
the series \eqref{1-1} can be continued meromorphically.
Therefore we may consider the limit values at non-positive integer points, and
in this case, sometimes the limit value is not included in the field generated by relevant
parameters.

Now consider the twisted case $k\geq 1$.   
When $k=n$, the series \eqref{1-1} is fully twisted, and the situation is now well-understood. It is shown by de Crisenoy \cite{dC} that in this case, when the polynomials
$P_j$ satisfy the HDF condition (see Definition \ref{def_HDF} below), 
the series \eqref{1-1} can be continued to an
entire function. He also proved an explicit expression of non-positive integer values, which will prove to be very useful in this paper.

If $1\leq k\leq n-1$, the situation may be called the ``partially twisted''
case.   Partially twisted double series are important in connection
with functional equations \cite{Mats04} \cite{KMT11} (see also
\cite{ChMa16} \cite{ChMa17}).
General $n$-fold partially twisted case was first studied by Komori \cite{Komo10} when all $P_j$ are
linear forms.   He obtained an explicit formula for certain limit values, expressed in terms of
generalized multiple Bernoulli numbers (whose notion he introduced).   His main tool is 
a multi-dimensional contour integration.

Another approach to the partially twisted case was developed by Essouabri and 
Matsumoto \cite{CMUSP}, 
in which the cases when $T=n$ and $P_j$ are linear forms or power sums are studied.
When $k=n-1$, by using the Mellin-Barnes integral formula it is possible to reduce to the
de Crisenoy case \cite{dC}; it can be found that non-positive integer points are regular points, and
an explicit formula for the values at those points are obtained.   Next, the case $k=n-2$ 
can be reduced to the case $k=n-1$, and an explicit formula for limit values (because
singular points appear for $k=n-2$) is shown.    In \cite{CMUSP} only these two cases are
discussed, but in the same way we can study the cases $k=n-3, n-4$ and so on inductively.

The aim of the present paper is to study the more general situation.   In fact,
since only the cases when $P_j$ are linear forms or power sums are studied in \cite{CMUSP},
in the present paper we aim to extend our consideration to the case when $T$ is not necessarily equal to $n$, and $P_j$ are more general polynomials.

Since the situation becomes more and more complicated, in the present paper we 
restrict ourselves to the case $k=n-1$.
Then we may rewrite \eqref{1-1} to
\begin{align}\label{1-4}
&\zeta_n(\mathbf{s};\mathbf{P};\boldsymbol{\mu}_{n-1})\notag\\
&\;=\sum_{m_1,\ldots,m_{n-1}\geq 1}
\left(\frac{\mu_1^{m_1}\cdots \mu_{n-1}^{m_{n-1}}}
{\prod_{j=1}^{T-1}P_j(m_1,\ldots,m_{n-1})^{s_j}}\sum_{m_n\geq 1}\frac{1}{P_T(m_1,\ldots,m_n)^{s_T}}\right).
\end{align}
The main point in the following sections is how to separate $m_n$ from the other variables 
$m_1,\ldots,m_{n-1}$ in the inner sum.

The present paper is organized as follows.    Our main strategy is to use the (multiple)
Mellin-Barnes formula to carry out the separation process mentioned above.
However, before going into the main body of the paper, first of all in Sections \ref{preparation} and \ref{Comp_dC}
we collect several necessary tools and prove a complement we need of de Crisenoy's result \cite{dC}. Then we will start our Mellin-Barnes argument from Section \ref{section:case_d_equal_1}.
Hereafter we write $P_T$ as follows:
\begin{align}\label{1-5}
P_T(X_1,\ldots,X_n) = Q_0(X_1,\ldots,X_{n-1}) X_n^{a_0} + \cdots + Q_d(X_1,\ldots,X_{n-1}) X_n^{a_d}
\end{align}
with $d \in \mathbb{N}_0$, $a_d \neq 0$, $0 \leq a_0 < \cdots < a_d$.
In Section \ref{section:case_d_equal_1}, we consider the case $d=1$ and $a_0=0$.
In Section \ref{section:case_d_equal_2}, we consider the case $d=2$ and $a_0=0$.   These are very special cases, but the treatment in these sections gives the
prototype of our method.    Then in Section \ref{section:general_case} we discuss the general
situation.
Our main general results are stated as Theorem
\ref{th:analytic_continuation_zeta_general_case}, Corollary
\ref{cor:set_of_singularities_zeta_n_general_case}, and Theorem
\ref{th:value_zeta_n} in Section \ref{section:general_case}.
These theorems give the information of meromorphic continuation
and location of singularities of $\zeta_n(\mathbf{s};\mathbf{P};\boldsymbol{\mu}_{n-1})$
and an explicit formula for special values at non-positive integer points.
The proofs of those results will be given in Section \ref{section:proofs}.
Finally we mention some specific examples in Section \ref{application}.

It is to be stressed that, when the denominator of \eqref{1-4}
includes a polynomial of degree $\geq 2$, sometimes its special
values at non-positive integer points are of some transcendental nature.    As mentioned above, this kind of phenomenon was already observed in \cite{EM21}.    The same type of phenomenon happens
in the situation treated in the present paper; see 
Example \ref{transcendental} in Section \ref{application}.

Though in the present paper we only consider the case $k=n-1$, it is quite plausible that we may study the case $k\leq n-2$ in a similar way, as in \cite{CMUSP}.

\section{Necessary tools}
\label{preparation}

We begin with the definitions of conditions $\mathrm{H_0S}$ (for \textit{sufficient minimal hypothesis}) and HDF (for \textit{weak decreasing hypothesis}, or ``\textit{hypothèse de décroissance faible}" in French).

\begin{definition}\label{def_H0S}
We call that a polynomial $P\in \mathbb{R}[X_1,\ldots,X_n]$ satisfies the $\mathrm{H_0S}$ condition 
if, for any $\mathbf{x}\in [1,+\infty)^n$, it satisfies that
\begin{enumerate}[label=(\roman*)]
\item $P(\mathbf{x})>0$, and
\item $(\partial^{\boldsymbol{\alpha}}P/P)(\mathbf{x})\ll 1$ (for any $\boldsymbol{\alpha}\in \mathbb{N}_0^n$).
\end{enumerate}
\end{definition}

\begin{definition}\label{def_HDF}
We call that a polynomial $P\in \mathbb{R}[X_1,\ldots,X_n]$ satisfies the HDF condition 
if, for any $\mathbf{x}=(x_1,\ldots,x_n)\in [1,+\infty)^n$, it satisfies that
\begin{enumerate}[label=(\roman*)]
\item $P(\mathbf{x})>0$,
\item for any $1\leq j\leq n$, there exists an $\varepsilon_0>0$ such that 
$(\partial^{\boldsymbol{\alpha}}P/P)(\mathbf{x})\ll x_j^{-\varepsilon_0}$
for all $\boldsymbol{\alpha}=(\alpha_1,\ldots,\alpha_n)\in \mathbb{N}_0^n$
such that $\alpha_j\geq 1$.
\end{enumerate}
\end{definition}

The condition $\mathrm{H_0S}$ was first introduced by Essouabri \cite{Esso97}, and then, the stronger
condition HDF appeared in the paper \cite{dC} of de Crisenoy.
The following rather general class of fully twisted multiple series was studied by de Crisenoy \cite{dC}. The result of de Crisenoy on this fully twisted series is:
\begin{prop}
[{de Crisenoy, \cite[Theorems A \& B]{dC}}]
\label{thm_dC}
Let $R_1,\ldots,R_T \in \mathbb{R} [X_1,\ldots,X_n]$ satisfying the HDF condition, and 
\begin{align}\label{HDFadd}
\prod_{\rho=1}^T R_{\rho}(\mathbf{x})\to +\infty \quad \text{as} \quad
|\mathbf{x}|=x_1+\cdots+x_n\to +\infty, \;\mathbf{x}\in [1,+\infty)^n.
\end{align}
Then, the multiple zeta-function
\begin{equation}
    \zeta_{n}^{\dc}(\mathbf{s};\mathbf{R};\boldsymbol{\mu}_n) := \sum_{m_1,\ldots,m_n \geq 1} \frac{\mu_1^{m_1} \cdots \mu_n^{m_n}}{R_1(\mathbf{m})^{s_1} \cdots R_T(\mathbf{m})^{s_T}} \label{eq:def_zeta_de_crisenoy}
\end{equation}
can be extended to an entire function on $\mathbb{C}^T$.   Moreover, for any 
$\mathbf{k}=(k_1,\ldots,k_T)\in \mathbb{N}_0^T$, the explicit formula
\begin{align}\label{dC_result}
\zeta_{n}^{\dc}(-\mathbf{k};\mathbf{R};\boldsymbol{\mu}_n)=\sum_{\boldsymbol{\alpha}\in S}
a_{\boldsymbol{\alpha}}\prod_{j=1}^n \zeta_{\mu_j}(-\alpha_j)
\end{align}
holds,
where $S$ and $a_{\boldsymbol{\alpha}}$ are defined by the expansion
$$\prod_{\rho=1}^T R_{\rho}(\mathbf{X})^{k_\rho}=\sum_{\boldsymbol{\alpha}\in S}a_{\boldsymbol{\alpha}}
\mathbf{X}^{\boldsymbol{\alpha}}
$$
with $\mathbf{X}^{\boldsymbol{\alpha}}=X_1^{\alpha_1}\cdots X_n^{\alpha_n}$, and 
$\zeta_{\mu}$ is the Lerch zeta-function defined by \eqref{Lerch_def}.
\end{prop}

By convention in this paper, $n$ is allowed to be equal to $0$, and in that case the polynomials $R_1,\ldots,R_T$ correspond to real numbers $r_1,\ldots,r_T$, and $\zeta_{0}^{\dc}(\mathbf{s};\mathbf{R};\boldsymbol{\mu}_0) = r_1^{-s_1} \cdots r_T^{-s_T}.$

\begin{rem}
The Lerch zeta-function $\zeta_{\mu}$ has analytic continuation to $\C$ and its values at non positive integers can be expressed in term of Stirling numbers of the second kind (see, for example, \cite[Lemma 5.7]{dC}).
\end{rem}

In this paper we will use de Crisenoy's result. Moreover, we also need a complement to his result on the order of $\zeta_{n}^{\dc}(\mathbf{s};\mathbf{R};\boldsymbol{\mu})$.

\begin{definition}\label{moderategrowth}
A holomorphic function $F$ in a domain $U$ of $\C^T$ is said to be {\it of moderate growth} on $U$ if for any compact set $K$ of $U\cap \R^T$, there exists two constants $A=A(K,F), C=C(K,F) \in \R_{>0}$ such that 
$$F(s_1,\ldots, s_T) \leq C \, \left(1+\sum_{j=1}^T |\Im s_j|\right)^A$$
for all $(s_1,\ldots, s_T) \in U$ such that $\left(\Re s_1,\ldots, \Re s_T\right) \in K$.
\end{definition}

\begin{prop}\label{thm_dC_compl}
Assume that assumptions of Proposition \ref{thm_dC} hold. Then the entire function $\zeta_{n}^{\dc}(\mathbf{s};\mathbf{R};\boldsymbol{\mu}_n)$ is of moderate growth on $\C^T$.
\end{prop}

In the statement of his main theorem (Proposition \ref{thm_dC} above), de Crisenoy did not mention that the holomorphic continuation has moderate growth. Since we need this fact in our proofs in this paper, we will give in Section \ref{Comp_dC} a proof of Proposition \ref{thm_dC_compl} based in part on de Crisenoy's method.

Our tool of reducing the partially twisted case to the above theorem on the fully twisted
case is the following Mellin-Barnes type formulas.   First we state the simplest 
form (see \cite[Section 14.51]{WW}):

\begin{prop}
[The Mellin-Barnes formula]
\label{MB_formula}
Let $s,\lambda \in \C$ with $\Re s>0$, $|\arg (\lambda)|<\pi$ and $\lambda \neq 0$. Then, we have 
$$(1+\lambda)^{-s}=\frac{1}{2\pi \ii} \int_{(c)}\frac{\Gamma (s+z) \Gamma (-z)}{\Gamma (s)}\lambda^z \, \dd z,$$
where $c$ is a real number with $-\Re s < c < 0$ and the path $(c)$ of integration is the vertical line $\Re z=c$.
\end{prop}

Next we quote the multiple version due to Mellin \cite{Mell}, and
was successfully used in Essouabri \cite{Esso12}.

\begin{prop}
[The multiple Mellin-Barnes formula]
\label{MMB_formula}
Let $s,\lambda_0,\ldots, \lambda_r \in \C$. Let $\rho_1, \ldots, \rho_r >0$. Assume that 
$\Re \lambda_j>0$ for any $j=0,\ldots, r$ and that $\Re s>\rho_1+\cdots+\rho_r$. Then, 
\begin{align*}
&\left(\lambda_0+\cdots +\lambda_r\right)^{-s}\\
&\quad=
\frac{1}{(2\pi \ii)^r} \int_{(\rho_1)} \cdots \int_{(\rho_r)}
\frac{\Gamma (s-z_1-\cdots-z_r) \prod_{j=1}^r\Gamma (z_j)}
{\Gamma (s)\lambda_0^{s-z_1-\cdots -z_r} \prod_{j=1}^r\lambda_j^{z_j}}
\, \dd z_1 \cdots \dd z_r.
\end{align*}
\end{prop}

We conclude this section with the following result on the domain of absolute convergence of the series \eqref{1-1}.
In the sequel of this paper we will assume that 
$P_1(\mathbf{m}),\ldots, P_{T-1}(\mathbf{m}), Q_0(\mathbf{m}),\ldots, Q_d(\mathbf{m}) >0$ for all $\mathbf{m} \in \N^{n-1}$ and 
\begin{equation}\label{eq:condition_polynomials_tends_infinity}
P_1 \cdots P_{T-1}
\left(Q_0 \cdots Q_d\right)^{\frac{1}{d+1}}(x_1,\ldots, x_{n-1}) \to \infty, 
\end{equation}
as $x_1+\cdots+x_{n-1} \to \infty$ $\left((x_1,\ldots, x_{n-1}) \in [1, \infty)^{n-1}\right)$.

By the AM–GM inequality, we get
$$(P_1\cdots P_T) (x_1,\ldots ,x_n) \gg_{\mathbf{P}} P_1\cdots P_{T-1} (Q_0\cdots Q_d)^{\frac{1}{d+1}} (x_1,\ldots,x_{n-1}) x_n^a, $$
where $a:=\frac{a_0+\cdots+a_d}{d+1}$.
Assumption \eqref{eq:condition_polynomials_tends_infinity} implies then by using \cite[Lemma 1]{Esso97} that there exist $r_0, b >0$ such that for all $\mathbf{x}\in [1,\infty)^n$ such that $\lVert \mathbf{x} \rVert \geq r_0$,
$$(P_1\cdots P_{T})(x_1,\ldots,x_n) \gg_{\mathbf{P}} (x_1\cdots x_{n-1})^b x_n^a. $$
Therefore we obtain:
\begin{prop}\label{conv_domain}
The multiple series \eqref{1-1} converges in the domain 
\begin{align}\label{1-2}
\Re s_1,\ldots,\Re s_{T-1}> 1/b \quad \text{ and } \quad \Re s_T> 1/a.
\end{align}
\end{prop}

\section{Proof of Proposition \ref{thm_dC_compl}}
\label{Comp_dC}

\subsection{Several definitions}
 
 \begin{definition}\label{classBr} 
 (\cite[Definition 2.3]{dC})
 For $r\in \R$, we define $\mathcal B (r)$ as the set of function $f: [r, \infty)\to \C$ such that there exists a sequence 
 $(f_n)_{n\geq 0}$ of $\mathcal C^\infty$ bounded function on $[r, \infty)$ verifying $f_0=f$ and $f_{n+1}'=f_n$ for any $n\in \N_0$.
 \end{definition}

\begin{example}
(\cite[Example 2.6]{dC})
If $\alpha, \beta \in \R$ and $a\in \C$ satisfy $\beta \neq 0$, $\alpha/\beta \not \in \Z$ and $|a|\neq 1$, then for any $r\in \R$, the function $f : [r, \infty)\to \C$ defined by $\ds f(x)=\frac{\phie^{\ii \alpha x}}{1-a \phie^{\ii \beta x}}$, belongs to $\mathcal B(r)$.    
\end{example}
We extend Definition \ref{classBr} as follows:
 \begin{definition}\label{classBrM+} 
 For $r\in \R$ and $M\in \R_+$, we define $\mathcal B (r;M)$ as the set of function $f: [r, \infty)\to \C$ such that there exists a sequence 
 $(f_n)_{n\geq 0}$ of $\mathcal C^\infty$ functions on $[r, \infty)$ satisfying $f_0=f$, $f_{n+1}'=f_n$ and 
 $\sup_{x\geq r}|f_n(x)|\leq M$ for any $n\in \N_0$.
 \end{definition}

\begin{definition}\label{entirecomb+}
Let $Y, Y_1,\ldots, Y_k$ functions defined in an open subset $U$ of $\C^T$, we say that $Y$ is an {\it entire combination} of $Y_1,\ldots, Y_k$ if there exists entire functions $\lambda, \lambda_1, \ldots, \lambda_k :\C^T\to \C$ such that $Y=\lambda + \sum_{i=1}^k \lambda_i Y_i$ on $U$. 
\end{definition}

In the proof of \cite[Theorem 2.7]{dC}, de Crisenoy introduced the following spaces:

\begin{definition}\label{EuQ}
For $\u=(u_t)_{t=1}^T \in \N_0^T$, the set $\mathcal E_{\u}(Q)$ is defined as the vector subspace of $\C[X_1,\ldots,X_n]$ generated by all the polynomials of the form $\ds \partial^{\betab} Q \prod_{t=1}^T \prod_{k\in F_t} \partial^{\varphi_t(k)}R_t$, where
\begin{enumerate}
\item $\betab \in \N_0^n$;
\item the $F_t$ are finite and pairwise disjoint subsets of $\N_0$, satisfying $|F_t|=u_t$;
\item for all $t=1,\ldots, T$, $\varphi_t$ is a function from $F_t$ to $\N_0^n$ such that there exists pairwise disjoint subsets $D_1,\ldots, D_n$ of $\N_0^n$,  satisfying:
\begin{enumerate}[label=(\alph*)]
\item $|D_1|=\cdots=|D_n|$;
\item $\sqcup_{k=1}^n D_k = \sqcup_{t=1}^T F_t$;
\item for any $t\in \{1,\ldots,T\}$, any $k\in \{1,\ldots,n\}$, and any $j\in D_k\cap F_t$, it holds that
$\varphi_t(k) \in \N_0^{k-1}\times \N\times \N_0^{n-k}$.
\end{enumerate}
\end{enumerate}
\end{definition}

\subsection{A uniform estimate of a certain integral}

By using in part de Crisenoy's proof, we obtain the following refinement of his Theorem 2.7 in \cite{dC}:
\begin{prop}\label{Integralmoderategrowth}
Let $\delta, M, C_1, C_2, C_3,\eps_0>0$ and $p,q \in \N$.
Let $Q, R_1,\ldots, R_T\in \C[X_1,\ldots, X_n]$ and 
$n_1\in \{0,\ldots, n\}$. We assume that:
\begin{enumerate}[label=(\roman*)]
\item $\max_{k=1,\ldots, n} \deg_{X_k} Q \leq q$;
\item $\max \{ \deg_{X_k }R_t \mid t=1,\ldots, T \text{ and }
k=1,\ldots, n\} \leq p$;
\item $R_t(\x) \not \in \R_{\leq 0}$
for any $t=1,\ldots, T$ and any $\x \in [-1,+1]^{n_1}\times [1,\infty)^{n-n_1}$;
\item $|R_t(\x)| \geq C_1$ for any $t=1,\ldots,T$ and
 any $\x \in [-1,+1]^{n_1}\times [1,\infty)^{n-n_1}$;
\item $\prod_{t=1}^T |R_t(\x) | \to \infty$ as $|\x|\to \infty$ $(\x \in [-1,+1]^{n_1}\times [1,\infty)^{n-n_1})$;
\item let $\alphab \in \{\zerob\}^{n_1}\times \N_0^{n-n_1}$ and $k\in \{n_1+1,\ldots,n\}$, and if $\alpha_k \geq 1$, then 
$$\left|\frac{\partial^{\alphab} R_t}{R_t} (\x)\right| \leq C_2 x_k^{-\eps_0}$$
for any $t=1, \ldots,T$ and any $\x \in [-1,+1]^{n_1}\times [1,\infty)^{n-n_1}$;
\item the coefficients of the polynomials $Q, R_1,\ldots, R_T$ are bounded by $C_3$;
\item $|\arg (R_t(\x))| \leq \delta$
for any $t=1,\ldots, T$ and any $\x \in [-1,+1]^{n_1}\times [1,\infty)^{n-n_1}$.  
\end{enumerate}
Let $f_{n_1},\ldots, f_n \in \mathcal B(1; M)$ and $f:[-1,+1]^{n_1}\to \C$ a continuous function such that $\sup_{\x\in [-1+1]^{n_1}}|f(\x)|\leq M$. Define:
\begin{align*}
Y(\s)=&Y^{n_1, n-n_1}(Q; R_1,\ldots, R_T;f, f_{n_1+1},\ldots,f_n; \s)\\
:=& \int_{ [-1,+1]^{n_1}\times [1,\infty)^{n-n_1}}Q(\x) \left(\prod_{t=1}^T R_t(\x)^{-s_t}\right) 
f(x_1,\ldots, x_{n_1})\\
& \qquad \times \left(\prod_{k=n_1+1}^n f_k (x_k)\right) \, \dd x_1\cdots \dd x_n.
\end{align*}
Then, 
\begin{enumerate}
\item $Y$ has an analytic continuation to $\C^T$ as an entire function;
\item for any $a>0$, there exists  $A=A(a,p,q, \eps_0)>0$ such that we have 
\begin{equation}\label{crucialestimate}
Y(\s) \ll_{a,n_1,n, T, p, q, \eps_0, C_1, C_2, C_3} M^n \left(1+\sum_{t=1}^T |s_t| \right)^A \phie^{\delta \left(\sum_{t=1}^T |\Im s_t|\right)}
\end{equation}
uniformly in $\s \in \C^T$ such that $\Re s_t >-a$ for all $t=1,\ldots,T$.
\end{enumerate}
\end{prop}

\begin{rem}
The main point in the estimate (\ref{crucialestimate}) is that the implicit constant depends at most on $a,n_1,n, T, p, q, \eps_0, C_1, C_2$ and $C_3$. So, Theorem \ref{Integralmoderategrowth} can be seen as a uniform version of de Crisenoy's Theorem 2.7 in \cite{dC}.
\end{rem}

\begin{proof}[Proof of Proposition \ref{Integralmoderategrowth}]
The first point follows from de Crisenoy's work (\cite[Theorem 2.7]{dC}).

To prove the second point we assume for simplicity that $n_1=0$ and $f\equiv 1$. The proof for $n_1 \geq 1$ is similar.

Let $Q, R_1,\ldots, R_T\in \C[X_1,\ldots, X_n]$ and $f_1,\ldots, f_n \in \mathcal B(1, M)$ as in the statement of the theorem.  It is clear that there exists $\sigma_0>0$ such that
\begin{align*}
&Y^{0, n}(Q; R_1,\ldots, R_T; f_1,\ldots,f_n; \s)\\
&\qquad = \int_{ [1,\infty)^n}Q(\x) \left(\prod_{t=1}^T R_t(\x)^{-s_t}\right) 
\left(\prod_{k=1}^n f_k (x_k)\right) \, \dd \x
\end{align*}
converges if $\Re s_t>\sigma_0$ for all $t=1,\ldots, T$.

The proof of the proposition proceeds by induction on $n$.
Further we need another induction argument:
de Crisenoy proved by induction on $m\in \N$ (see Step 9 in p.~1385 of \cite{dC}) that 
$Y^{0, n}(Q; R_1,\ldots, R_T; f_1,\ldots,f_n; \s)$ is an entire combination of functions of the form 
$$Y^{0, n}(G; R_1,\ldots, R_T; g_1,\ldots,g_n; \s+\u),$$ where $\u \in \N_0^T$, $|\u|=m n$, $G\in \mathcal E_{\u}(Q)$, $g_1,\ldots, g_n \in \mathcal B (1)$. 

Since at each stage of his recurrence, de Crisenoy proceeds only by integration by parts, it is clear that his proof implies that the coefficients of this combination are product of polynomial of $\s$ of degree $O(m)$ (which contributes to the 
$(1+\sum_t |s_t|)^A$ factor in \eqref{crucialestimate})
and function of the form 
$$Y^{0, n'}(H; R_1',\ldots, R_T'; h_1,\ldots,h_n; \s+\u)$$ 
where $n'\leq n-1$, $H\in \C[X_1,\ldots,X_{n'}]$ with coefficients of size $O(C_3)$, and each $R_t'$ is obtained from $R_t$ by substituting the value $1$ for certain variables $X_i$, and $h_1,\ldots,h_{n'} \in \mathcal B(1)$.
Since each of the functions $g_i$ and $h_i$ is an element of a sequence associated to some $f_k$ by Definition \ref{classBr}, we deduce from our definition of $\mathcal B(1,M)$ that, as $f_k$, the functions $g_i$ and $h_i$ are also elements of the set 
$\mathcal B(1,M)$.
Thus, by induction hypothesis, to conclude it is enough to prove that for $m\in \N$ sufficiently big and depending only on $a, p, q , \eps_0$, for $\u \in \N_0^T$ with $|\u|=m n$, 
for $G\in \mathcal E_{\u}(Q)$ and $g_1,\ldots, g_n \in \mathcal B (1, M)$, the function 
$$ \s\mapsto Y^{0, n}(G; R_1,\ldots, R_T; g_1,\ldots,g_n; \s+\u) $$ 
is holomorphic in the domain $\{\s \in \C^T\mid \Re s_t>-a~ (t=1,\ldots, T)\}$ and to verify the estimate (\ref{crucialestimate}) for this $Y^{0,n}$.

Since $G\in \mathcal E_{\u}(Q)$, it follows from the definition of $\mathcal E_{\u}(Q)$ that it is enough to study the case where
$G= \partial^{\betab} Q \prod_{t=1}^T \prod_{k\in F_t} \partial^{\varphi_t(k)}R_t$ as in Definition \ref{EuQ}. 
Since $|\u|= m n$, we have 
\begin{align}\label{order_D_k}
|D_k|=m \quad(k=1,\ldots,n),
\end{align}
because, noting 3-(a), 3-(b) of Definition \ref{EuQ},
$$
|\u|=u_1+\cdots+u_T=|F_1|+\cdots+|F_T|
=|D_1|+\cdots+|D_n|=n|D_k|
$$
for any $k$.

It then follows that we have uniformly in $\x\in  [1,\infty)^n$,
\begin{align*}
|G(\x)|&= |\partial^{\betab} Q(\x)| \prod_{t=1}^T \prod_{j\in F_t} |\partial^{\varphi_t(j)}R_t(\x)|\\
&= |\partial^{\betab} Q (\x)| \prod_{t=1}^T\prod_{k=1}^n \prod_{j\in F_t\cap D_k} |\partial^{\varphi_t(j)}R_t(\x)|\\
&\ll_{C_2, T, n} |\partial^{\betab} Q (\x)| \prod_{t=1}^T\prod_{k=1}^n \prod_{j\in F_t\cap D_k}|x_k^{-\eps_0} R_t(\x)|,
\end{align*}
where the last inequality follows by noting 
3-(c) of Definition \ref{EuQ} and (vi) of Proposition
\ref{Integralmoderategrowth}.
The above is further estimated as
\begin{align}\label{ggggg}
&\ll_{C_2, T, n} |\partial^{\betab} Q (\x)| \prod_{t=1}^T\prod_{k=1}^n \left(x_k^{-\eps_0} |R_t(\x)|\right)^{|F_t\cap D_k|}\nonumber\\
&\ll_{C_2, T, n} |\partial^{\betab} Q (\x)| \left(\prod_{k=1}^n x_k^{-\eps_0 |D_k|}\right) \left(\prod_{t=1}^T |R_t(\x)|^{|F_t|}\right)\nonumber\\
&\ll_{C_2, T, n} |\partial^{\betab} Q (\x)| \left(\prod_{k=1}^n x_k^{-\eps_0 m}\right) \left(\prod_{t=1}^T |R_t(\x)|^{u_t}\right)
\end{align}
by \eqref{order_D_k}.

Let $a >0$. Choose now $\ds m:=\left\lfloor \frac{q+aTp +2}{\eps_0}\right\rfloor +1\in \N$. 
Then we have 
$\displaystyle{\frac{q+aTp+2}{\eps_0}<{m}}$,
so $-\eps_0 m +q+aT p< -2$.   Set
\begin{align*}
U(\x):=&G(\x) \prod_{t=1}^T R_t^{-s_t-u_t} (\x) \prod_{k=1}^nf_k(x_k),\\
V:=&\{a,n_1,n, T, p, q, \eps_0, C_1, C_2, C_3\}.
\end{align*}
It follows from \eqref{ggggg} that we have uniformly in  $\x\in  [1,\infty)^n$, in $\s \in \C^T$ such that 
$\Re s_t>-a$ for all $t=1,\ldots, T$, and in $f_1,\ldots, f_n \in \mathcal B (1, M)$:
\begin{align*}
U(\x)&\ll_V M^n |\partial^{\betab} Q (\x)| \left(\prod_{k=1}^n x_k^{-\eps_0 m}\right) \left(\prod_{t=1}^T |R_t(\x)|^{-\Re s_t}\right)\\
& \qquad \times \exp \left(\sum_{t=1}^T |\Im s_t| \, |\arg (R_t(\x))|\right)\\
&\ll_V M^n |\partial^{\betab} Q (\x)| \left(\prod_{k=1}^n x_k\right)^{-\eps_0 m} \left(\prod_{t=1}^T |R_t(\x)|\right)^a\\
& \qquad \times \exp \left(\delta \sum_{t=1}^T |\Im s_t| \right)\\
&\ll_V M^n \left(\prod_{k=1}^n x_k\right)^{-\eps_0 m +q+a T p} \exp \left(\delta \sum_{t=1}^T |\Im s_t| \right)\\
&\ll_V  M^n \left(\prod_{k=1}^n x_k\right)^{-2} 
\exp \left(\delta \sum_{t=1}^T |\Im s_t| \right),
\end{align*}
where we use (viii) of Proposition
\ref{Integralmoderategrowth} for the second inequality, 
(i) and (ii) of Proposition \ref{Integralmoderategrowth} for the third inequality.
It follows then that 
\begin{align}\label{ggggg2}
Y^{0,n}(G; R_1, \ldots, R_T; g_1,\ldots,g_n; \s+\u)&=\int_{[1,\infty)^n} U(\x) \, \dd \x\notag\\
&\ll_V M^n \exp \left(\delta \sum_{t=1}^T |\Im s_t| \right),
\end{align}
which implies the second point of Proposition \ref{Integralmoderategrowth}. This ends the proof of Proposition \ref{Integralmoderategrowth}.
\end{proof}

\subsection{Completion of the proof of Proposition \ref{thm_dC_compl}}

Now we complete the proof of Proposition 
\ref{thm_dC_compl}.
Let $R_1,\ldots,R_T \in \mathbb{R} [X_1,\ldots,X_n]$ satisfying assumptions of Proposition \ref{thm_dC}, and $\boldsymbol{\mu}_n=(\mu_1,\ldots,\mu_{n})\in (\mathbb{T} \setminus \{ 1 \})^{n}$.

By induction it is enough to prove Proposition \ref{thm_dC_compl} for the function
$$Z^*(\s):=\sum_{m_1,\ldots,m_n \geq 2} \frac{\mu_1^{m_1}\cdots \mu_n^{m_n}}{R_1(\mathbf{m})^{s_1} \cdots R_T(\mathbf{m})^{s_T}}
$$
(see \cite[pp.~1390--1391]{dC}).
Proposition \ref{conv_domain} implies that there exists $\sigma_0>0$ such that $Z^*(\s)$
converges for $\mathbf{s}=(s_1,\ldots,s_T)$ such that $\Re s_t>\sigma_0$ for all $t=1,\ldots,T$.

For $k=1,\ldots, n$, fix $\theta_k \in \R\setminus 2\pi \Z$ such that $\mu_k =\phie^{\ii \theta_k}$.
For $P\in \C[X_1,\ldots, X_n]$ and $n_1,n_2,n_3\in \N_0$ such that $n_1+n_2+n_3=n$ and $\eps >0$, we define $P_\eps^{n_1, n_2, n_3} \in \C[X_1,\ldots, X_n]$ by
\begin{align*}
& P_\eps^{n_1, n_2, n_3}(x_1,\ldots,x_n)\\
& :=P\left(\frac{3}{2}+i\eps x_1,\ldots,\frac{3}{2}+i\eps x_{n_1},x_{n_1+1}+i\eps,\ldots,x_{n_1+n_2}+i\eps,\right.\\
&\biggl.\qquad\qquad x_{n_1+n_2+1}-i\eps,\ldots,x_{n}-i\eps\biggr)\\
&= P\left(\left(\frac{3}{2},\ldots,\frac{3}{2}, ,x_{n_1+1},\ldots, x_{n}\right)+i\eps\left(x_1,\ldots,x_{n_1},1,\ldots,1,-1,\ldots,-1\right)\right).
\end{align*}
By using the method of \cite{Esso97}, de Crisenoy proved in \cite[pp.~1390--1391]{dC} that there exists $\eps_1>0$ such that for $\eps \in (0,\eps_1)$, $Z^*(\s)$ is a linear combination of integrals of the form
\begin{align*}
Y_\eps^{n_1,n_2,n_3}(\s)
&:= \int_{ [-1,+1]^{n_1}\times [3/2,\infty)^{n-n_1}}\left(\prod_{t=1}^T 
(R_t)_\eps^{n_1,n_2,n_3}(\x)^{-s_t}\right) \\
& \qquad  \times \left(\prod_{k=1}^n f_k (x_k)\right) \, \dd x_1\cdots \dd x_n,
\end{align*}
where 

\begin{itemize}
\item[$\bullet$] for $k=1,\ldots, n_1$, $f_k:[-1,+1]\to \C$ is defined by 
$$f_k(x):=-\phie^{\frac{3}{2} \ii \theta_k}
\frac{\phie^{-\eps\theta_k x}}{1+\phie^{-2\pi \eps x}},$$
and so $f:[-1,+1]^{n_1}\to \C$ defined by $f(x_1,\ldots, x_{n_1})=\prod_{k=1}^{n_1} f_k(x_k)$, is a continuous bounded function and satisfies $\sup_{[-1,+1]^{n_1}}|f|\ll 1$ uniformly in $\eps \in (0,\eps_1)$.
\item[$\bullet$] for $n_1+1\leq k \leq n_1+n_2$, $f_k:[3/2,\infty)\to \C$ is defined by 
$$f_k(x):=-\phie^{-\eps \theta_k}
\frac{\phie^{\ii\theta_k x}}{1-\phie^{-2\pi \eps}\phie^{2\pi \ii x}}.$$
\item[$\bullet$] for $n_1+n_2+1\leq k \leq n$, $f_k:[3/2,\infty)\to \C$ is defined by 
$$f_k(x):=-\phie^{\eps \theta_k}
\frac{\phie^{\ii\theta_k x}}{1-\phie^{2\pi \eps}\phie^{2\pi \ii x}}.$$
Since $\theta_k \not \in 2\pi \Z$, it follows that 
$f_k \in \mathcal B(3/2, M_\eps)$ where $M_\eps =O(\eps^{-1})$ uniformly in $\eps \in (0,\eps_1)$.
\end{itemize}

Using in addition Lemma 2 of \cite{Esso97}, we  easily verify the hypotheses of the Theorem \ref{Integralmoderategrowth} above with $M=M_\eps =O(\eps^{-1})$ and $\delta = O(\eps)$ and deduce from it $Z^*(\s)$ is an entire function and satisfies for any $a>0$, there exists $A=A(a, \mathbf{R})$ such that 
$$Z^*(\s) \ll_{\mathbf{R}} \eps^{-n} \left(1+\sum_{t=1}^T |s_t| \right)^A \phie^{O(\eps) \left(\sum_{t=1}^T |\Im (s_t)|\right)}$$
uniformly in $\s \in \C^T$ with $\Re s_t>-a$ for all $t=1,\ldots, T$ and $\eps \in (0,\eps_1)$.
By choosing $\ds \eps =\frac{\eps_1}{2+\left(\sum_{t=1}^T |\Im s_t|\right)}$, we conclude that 
$$Z^*(\s) \ll_{\mathbf{R}} \left(1+\sum_{t=1}^T |s_t| \right)^{A+n} $$
uniformly in $\s \in \C^T$ with $\Re s_t>-a$ for all $t=1,\ldots, T$.
This implies that $Z^*(\s)$ has moderate growth and ends the proof of Proposition \ref{thm_dC_compl}. \qed

\section{The case \texorpdfstring{$d=1$}{d equal 1} and \texorpdfstring{$a_0=0$}{a0 equal 0}}
\label{section:case_d_equal_1}

Now we are going into the main body of the present paper, that is, the study of $\zeta_n(\mathbf{s};\mathbf{P};\boldsymbol{\mu}_{n-1})$.    Before stating our general results, in this and the next sections, we will discuss some simple special cases.
In this section we consider the case when $P_T(X_1,\ldots,X_n)$ is of the form
\begin{align}\label{P_t_form1}
P_T(X_1,\ldots,X_n)=Q_0(X_1,\ldots,X_{n-1})+Q_1(X_1,\ldots,X_{n-1})X_n^{a_1},
\end{align}
where
$P_1,\ldots, P_{T-1}, Q_0,Q_1\in \mathbb{R}[X_1,\ldots,X_{n-1}]$ satisfy HDF and \eqref{eq:condition_polynomials_tends_infinity}, and $a_1$ is a positive integer. 
Then
\begin{align*}
&P_T(m_1,\ldots,m_n)^{-s_T}=
(Q_0(m_1,\ldots,m_{n-1})+Q_1(m_1,\ldots,m_{n-1})m_n^{a_1})^{-s_T}\\
&\quad =Q_0(m_1,\ldots,m_{n-1})^{-s_T}\left(1+
\frac{Q_1(m_1,\ldots,m_{n-1})m_n^{a_1}}{Q_0(m_1,\ldots,m_{n-1})}\right)^{-s_T}.
\end{align*}
Applying the Mellin-Barnes formula (Proposition \ref{MB_formula}), we see that the above is equal to
\begin{equation}\label{2-1}
Q_0(m_1,\ldots,m_{n-1})^{-s_T}\cdot \frac{1}{2\pi \ii}\int_{(c)}
\frac{\Gamma(s_T+z)\Gamma(-z)}{\Gamma(s_T)} \left(\frac{Q_1(m_1,\ldots,m_{n-1})m_n^{a_1}}{Q_0(m_1,\ldots,m_{n-1})}
\right)^z \, \dd z
\end{equation}
where the path of integration is the vertical line from $c-i\infty$ to
$c+i\infty$, and $c$ satisfies $-\Re s_T<c<0$.
But because of \eqref{1-2}, we may assume that
\begin{align}\label{2-2}
&\Re s_j >\frac{1}{b} \quad \text{ for any } j=1,\ldots, T-1,\notag\\
&\Re s_T >\frac{1}{b}-c \quad \text{ and } \quad c< \min\left(\frac{-1}{a}, \frac{-1}{a_1}\right).
\end{align}
Inserting \eqref{2-1} into \eqref{1-4}, we obtain
\begin{align}\label{2-3}
&\zeta_n(\mathbf{s};\mathbf{P};\boldsymbol{\mu}_{n-1})\notag\\
&\;=\sum_{m_1,\ldots,m_{n-1}\geq 1}
\left(\frac{\prod_{j=1}^{n-1}\mu_j^{m_j}}{\prod_{j=1}^{T-1}P_j(m_1,\ldots,m_{n-1})^{s_j}}\right)\notag\\
&\qquad \times\sum_{m_n\geq 1}
Q_0(m_1,\ldots,m_{n-1})^{-s_T} \notag\\
&\qquad \qquad \times \frac{1}{2\pi \ii}\int_{(c)}
\frac{\Gamma(s_T+z)\Gamma(-z)}{\Gamma(s_T)} \left(\frac{Q_1(m_1,\ldots,m_{n-1})m_n^{a_1}}{Q_0(m_1,\ldots,m_{n-1})}
\right)^z \, \dd z\notag\\
&\;=\frac{1}{2\pi \ii}\int_{(c)}
\frac{\Gamma(s_T+z)\Gamma(-z)}{\Gamma(s_T)}
\zeta_{n-1}^{\dc}(\mathbf{s}(z);\mathbf{P(Q)};\boldsymbol{\mu}_{n-1})\sum_{m_n\geq 1}m_n^{a_1 z} \, \dd z,
\end{align}
where
$\mathbf{s}(z)=(s_1,\ldots,s_{T-1},s_T+z,-z)$ and
$\mathbf{P(Q)}=(P_1,\ldots,P_{T-1},Q_0,Q_1)$.
Here, the sum with respect to $m_n$ is separated as 
$$
\sum_{m_n\geq 1}m_n^{a_1 z}=\zeta(-a_1 z).
$$
Note that, since $\Re z=c$, by \eqref{1-2} and \eqref{2-2} we see that
$\Re s_j>1$ for any $j=1,\ldots, T-1$, $\Re(-z)=-c>\frac{1}{a}$ and
$\Re(s_T+z)=\Re s_T+c>\frac{1}{b}$, 
hence $\zeta_{n-1}^{\dc}(\mathbf{s}(z);\mathbf{P(Q)};\boldsymbol{\mu}_{n-1})$ is an absolutely convergent series.
In view of Proposition \ref{thm_dC}, we see that
this series can be continued to an entire function.


Now we shift the path of integration from $\Re z=c$ to $\Re z=M-\eta$, where $M$ is
a large positive integer and $0<\eta<1$.    We encounter the following poles of the integrand:
\begin{enumerate}[label=(\roman*)]
\item $z=-1/a_1$ from the zeta factor,
\item $z=0,1,2,\ldots,M-1$ from the factor $\Gamma(-z)$.
\end{enumerate}
The poles of $\Gamma(s_T+z)$ are irrelevant.
Since
$$
\lim_{z\to -1/a_1}(z+1/a_1)\zeta(-a_1z)=\lim_{\delta\to 0}\delta\zeta(-a_1(-1/a_1+\delta))=-1/a_1,
$$
the residue at $z=-1/a_1$ is
$$
\frac{\Gamma (s_T-1/a_1)\Gamma(1/a_1)}{\Gamma(s_T)}\zeta_{n-1}^{\dc}(\mathbf{s}(-1/a_1),\mathbf{P(Q)},
\boldsymbol{\mu}_{n-1})\cdot (-1/a_1).
$$
The residue at $z=\ell$ ($0\leq\ell\leq M-1$) is
\begin{align*}
&\frac{\Gamma(s_T+\ell)}{\Gamma(s_T)}\frac{(-1)^{l+1}}{l!}
\zeta_{n-1}^{\dc}(\mathbf{s}(\ell);\mathbf{P(Q)};\boldsymbol{\mu}_{n-1})\zeta(-a_1\ell)\\
&\qquad=-\binom{-s_T}{\ell}\zeta_{n-1}^{\dc}(\mathbf{s}(\ell);\mathbf{P(Q)};\boldsymbol{\mu}_{n-1})\zeta(-a_1\ell).
\end{align*}
Therefore, after the shifting of the path, we have
\begin{align}\label{2-5}
&\zeta_n(\mathbf{s};\mathbf{P};\boldsymbol{\mu}_{n-1})
=\frac{\Gamma(1/a_1)}{a_1}\frac{\Gamma (s_T-1/a_1)}{\Gamma(s_T)}\zeta_{n-1}^{\dc}(\mathbf{s}(-1/a_1),\mathbf{P(Q)},
\boldsymbol{\mu}_{n-1})\notag\\
&\quad +\sum_{\ell=0}^{M-1}\binom{-s_T}{\ell}\zeta_{n-1}^{\dc}(\mathbf{s}(\ell);\mathbf{P(Q)};\boldsymbol{\mu}_{n-1})\zeta(-a_1\ell)
\notag\\
&\quad+\frac{1}{2\pi \ii}\int_{(M-\eta)}
\frac{\Gamma(s_T+z)\Gamma(-z)}{\Gamma(s_T)}
\zeta_{n-1}^{\dc}(\mathbf{s}(z);\mathbf{P(Q)};\boldsymbol{\mu}_{n-1})\zeta(-a_1z) \, \dd z.
\end{align}

The singularities appear only from the first part of 
the right-hand side.
If $a_1=1$, then
$$
\frac{\Gamma(s_T-1/a_1)}{\Gamma(s_T)}=\frac{\Gamma(s_T-1)}{\Gamma(s_T)}=\frac{1}{s_T-1},
$$
and we find that $s_T=1$ is the only singularity.    (This situation is similar to
the linear case treated in \cite{CMUSP}.)
When $a_1\geq 2$, the factor $\Gamma(s_T-1/a_1)$ produces infinitely many singular
hyperplanes.

In any case, when $s_T=-N_T$ is a non-positive integer, it is not on any singularity set.
Therefore, using \eqref{2-5}, we obtain an explicit expression of the values of
$\zeta_n(\mathbf{s};\mathbf{P};\boldsymbol{\mu}_{n-1})$ at non-positive integer points 
$\mathbf{s}=-\mathbf{N}=-(N_1,\ldots,N_T)$.
Because of the factor $\Gamma(s_T)$ on the denominator, for $s_T=-N_T$ the integral term always vanishes, and also
the first term vanishes except for the case $a_1=1$.
We obtain:

\begin{thm}\label{thm_d=1}
When $P_T$ is of the form \eqref{P_t_form1}, we have
\begin{align}\label{formula_d=1}
\zeta_n(-\mathbf{N};\mathbf{P};\boldsymbol{\mu}_{n-1})
=&\frac{-\delta_{1,a_1}}{N_T+1}\zeta_{n-1}^{\dc}(-\mathbf{N}(-1/a_1);
\mathbf{P(Q)};\boldsymbol{\mu}_{n-1})\notag\\
&+\sum_{\ell=0}^{N_T}\binom{N_T}{\ell}
\zeta_{n-1}^{\dc}(-\mathbf{N}(\ell);
\mathbf{P(Q)};\boldsymbol{\mu}_{n-1})\zeta(-a_1\ell),
\end{align}
where $\mathbf{N}(z)=(N_1,\ldots,N_{T-1},N_T+z,-z)$ and
$\delta_{1,a_1}$ is Kronecker's delta.
When $a_1\geq 2$, the explicit form of $\zeta_{n-1}$-factor 
on the right-hand side is given by \eqref{dC_result}.
\end{thm}

\section{The case \texorpdfstring{$d=2$}{d equal 2} and \texorpdfstring{$a_0=0$}{a0 equal 0}}
\label{section:case_d_equal_2}

In this section we consider the case when $P_T(X_1,\ldots,X_n)$ is of the form
\begin{align}\label{P_t_form2}
P_T(X_1,\ldots,X_n)=&Q_0(X_1,\ldots,X_{n-1})+Q_1(X_1,\ldots,X_{n-1})X_n^{a_1}\notag\\
&+Q_2(X_1,\ldots,X_{n-1})X_n^{a_2},
\end{align}
where
$P_1,\ldots, P_{T-1}, Q_0,Q_1, Q_2\in \mathbb{R}[X_1,\ldots,X_{n-1}]$ satisfy HDF and \eqref{eq:condition_polynomials_tends_infinity}, and $a_1, a_2$  positive integers such that $0<a_1<a_2$.

\begin{notation}
In this section we will use the notation $ \z'=(z_1,\ldots, z_{q-1})$ if $\z=(z_1,\ldots,z_q)$ ($q\in \N$). Similarly we will use the notation $\mathbf{s}'$, $\mathbf{m}'$, etc.
\end{notation}
We recall from (\ref{1-2}) that 
$\zeta_n(\mathbf{s};\mathbf{P};\mub_{n-1})$ converges absolutely if 
$$\Re s_1,\ldots,\Re s_{T-1}> 1/b \quad \text{ and } \quad \Re s_T> 1/a.$$
Set $\mathbf{P(Q)}=(P_1,\ldots,P_{T-1},Q_0,Q_1, Q_2)$. We know that there exists $\sigma_0>0$ such that 
$\zeta_{n-1}^{\dc}(s_1, \ldots, s_{T+2}; \mathbf{P(Q)};\boldsymbol{\mu}_{n-1})$ converges absolutely if $\Re (s_j)>\sigma_0$ for any $j=1,\ldots, T+2.$

Let $\rho_1, \rho_2 >\sigma_0$ such that $a_1\rho_1+a_2\rho_2 >1$.
Fix $\s=(\s', s_T) \in \C^{T-1}\times \C$ such that $\Re s_k >\max (\sigma_0,1/b)$ for any $k=1,\ldots, T-1$ and 
$\Re s_T >\max(1/a,\sigma_0+\rho_1+\rho_2) $. 
The multiple Mellin-Barnes formula (Proposition \ref{MMB_formula}) and the series-integral inversion imply that
\begin{align}\label{trucoublie}
&\zeta_n(\s;\P;\mub_{n-1})\notag\\
&=\sum_{m_1, \ldots, m_n\geq 1}
\frac{\mu_1^{m_1} \cdots \mu_{n-1}^{m_{n-1}}}{\left(\prod_{k=1}^{T-1} P_k(\m')^{s_k}\right)\left(Q_0(\m')+Q_1(\m')m_n^{a_1}+Q_2(\m')m_n^{a_2}\right)^{s_T}}\notag\\
&= \frac{1}{(2\pi \ii)^2} \int_{(\rho_1)}\int_{(\rho_2)}\frac{\Gamma (s_T-z_1-z_2) \Gamma (z_1) \Gamma (z_2)}{\Gamma (s_T)} 
\notag\\
& \times \sum_{m_1, \ldots, m_n\geq 1}\frac{\mu_1^{m_1} \cdots \mu_{n-1}^{m_{n-1}} \, \dd z_1 \, \dd z_2}{\left(\prod_{k=1}^{T-1} P_k(\m')^{s_k}\right) Q_0(\m')^{s_T-z_1-z_2} \left( Q_1(\m') m_n^{a_1}\right)^{z_1} 
\left( Q_2(\m') m_n^{a_2}\right)^{z_2}}\notag\\
&= \frac{1}{(2\pi \ii)^2} \int_{(\rho_1)}\int_{(\rho_2)}\frac{\Gamma (s_T-z_1-z_2) \Gamma (z_1) \Gamma (z_2)}{\Gamma (s_T)} 
\notag\\
& \times \sum_{m_1, \ldots, m_{n-1}\geq 1}\frac{\mu_1^{m_1}\cdots \mu_{n-1}^{m_{n-1}} \zeta \left( a_1 z_1+ a_2 z_2\right) }{\left(\prod_{k=1}^{T-1} P_k(\m')^{s_k}\right) Q_0(\m')^{s_T-z_1-z_2} Q_1(\m')^{z_1} Q_2(\m')^{z_2}} \, \dd z_1 \, \dd z_2 \notag\\
&= \frac{1}{(2\pi \ii)^2} \int_{(\rho_1)}\int_{(\rho_2)}\frac{\Gamma (s_T-z_1-z_2) \Gamma (z_1) \Gamma (z_2)}{\Gamma (s_T)} 
\notag\\
&  \quad \times \frac{\zeta_{n-1}^{\dc}(\s', s_T-z_1-z_2, z_1, z_2; \mathbf{P(Q)};\boldsymbol{\mu}_{n-1})
 G(z_1, z_2)}{ a_1 z_1+ a_2 z_2-1} \,
\dd z_1 \, \dd z_2,
\end{align}
where $\ds G(z_1, z_2) := \left( a_1 z_1+ a_2 z_2-1\right) \zeta \left(a_1 z_1+ a_2 z_2\right).$

Propositions \ref{thm_dC} and \ref{thm_dC_compl} imply that $\zeta_{n-1}^{\dc}(\cdot; \mathbf{P(Q)};\boldsymbol{\mu}_{n-1} )$ has holomorphic continuation of moderate growth to the whole space $\C^{T+2}$. In addition, it is clear from classical properties of Riemann zeta-function that $G$ has also holomorphic continuation of moderate growth to the whole space $\C^2$.
Therefore
we deduce by residue theorem and analytic continuation that (\ref{trucoublie}) holds if we assume that 
$\rho_1, \rho_2 >0, a_1 \rho_1 +a_2\rho_2 >1$ and $\s=(\s', s_T) \in \C^{T-1}\times \C$ such that $\Re s_k >1/b$ for any $k=1,\ldots, T-1$ and 
$\Re s_T >\max(1/a,\rho_1+\rho_2) $. 
We assume in the sequel that these assumptions hold. We assume also that $\rho_2<1/a_2$.

Let $M_1 \in \N_0$ and $\eps_1 \in (\rho_2,1)$.
Shift the path $(\rho_1)$ to $(-M_1-\eps_1)$. 
This shifting is possible, because 
$\zeta_{n-1}^{\dc}(\cdot; \mathbf{P(Q)};\boldsymbol{\mu}_{n-1})$ and $G$ are holomorphic of moderate growth, and gamma-factors are well estimated by Stirling's formula.
Relevant poles with respect to $z_1$ are
$z_1=-k_1$ ($0\leq k_1\leq M_1$) and 
$z_1=\frac{1-a_2 z_2}{a_1}$ (noting that by our assumptions we have 
$-M_1-\eps_1<\frac{1-a_2\rho_2}{a_1}<\rho_1$).
Therefore the residue theorem (in $z_1$) imply that 
\begin{align}\label{type12s}
\zeta_n(\s;\P;\mub_{n-1})
=\sum_{k_1=0}^{M_1}I_1(k_1)+I_2+
\frac{1}{\Gamma (s_T)}\mathcal R_{M_1+\eps_1}(\s),
\end{align}
where
\begin{align*}
&I_1(k_1)= \frac{(-1)^{k_1}}{2\pi \ii k_1!} \int_{(\rho_2)}\frac{\Gamma (s_T+k_1-z_2) \Gamma (z_2)}{\Gamma (s_T)} \\
& \quad\times  \zeta_{n-1}^{\dc}(\s', s_T+k_1-z_2, -k_1, z_2; \mathbf{P(Q)};\boldsymbol{\mu}_{n-1}) \frac{G(-k_1, z_2)}{a_2 z_2-a_1 k_1-1}\, \dd z_2,\\
&I_2= \frac{1}{2\pi \ii a_1} \int_{(\rho_2)}\frac{\Gamma \left(s_T+ \frac{(a_2-a_1)z_2-1}{a_1}\right) \Gamma (z_2)}{\Gamma (s_T)} \Gamma\left(\frac{1-a_2z_2}{a_1}\right)\\
& \times \zeta_{n-1}^{\dc}(\s', s_T+ \frac{(a_2-a_1)z_2-1}{a_1}, \frac{1-a_2z_2}{a_1}, z_2; \mathbf{P(Q)};\boldsymbol{\mu}_{n-1}) \, G(\frac{1-a_2z_2}{a_1}, z_2)\, \dd z_2,
\end{align*}
and
\begin{align*}
&\mathcal R_{M_1+\eps_1}(\s) =  \frac{1}{(2\pi \ii)^2} \int_{(-M_1-\eps_1)}\int_{(\rho_2)}\Gamma (s_T-z_1-z_2) \Gamma (z_1) \Gamma (z_2) \\
&\qquad  \times \zeta_{n-1}^{\dc}(\s', s_T-z_1-z_2, z_1, z_2; \mathbf{P(Q)};\boldsymbol{\mu}_{n-1}) \frac{G(z_1, z_2)}{a_1z_1+a_2z_2-1} \, \dd z_1 \, \dd z_2.
\end{align*}
Let $M_2 \in \N_0$ and $\eps_2 \in (0,1)$. We shift the path of the integral
in the definition of $I_1(k_1)$ from $(\rho_2)$ to $(-M_2-\eps_2)$, and that in the definition of $I_2$ from $(\rho_2)$ to  
$(M_2+\eps_2)$.   Since $\rho_2 \in (0, 1/a_2)$, we obtain by the residue theorem (in $z_2$) that  
\begin{align}\label{keyformula1}
&\zeta_n(\s;\P;\mub_{n-1})\notag\\
&=-\sum_{k_1=0}^{M_1}\sum_{k_2=0}^{M_2}\frac{(-1)^{k_1+k_2}}{k_1! k_2!}
\frac{\Gamma (s_T+k_1+k_2)}{\Gamma (s_T)} \notag \\
& \times \zeta_{n-1}^{\dc}(\s', s_T+k_1+k_2, -k_1, -k_2; \mathbf{P(Q)};\boldsymbol{\mu}_{n-1}) \frac{G(-k_1, -k_2)}{a_2k_2+a_1k_1+1}\notag\\
&  + \sum_{k_2=0}^{\left\lfloor\frac{a_2 M_2-1}{a_1}\right\rfloor}\frac{(-1)^{k_2}}{k_2! a_2}
\frac{\Gamma \left(s_T+\frac{(a_2-a_1)k_2-1}{a_2}\right) \Gamma\left(\frac{a_1 k_2+1}{a_2}\right)}{\Gamma (s_T)} \notag 
\\
& \times\zeta_{n-1}^{\dc}\left(\s', s_T+\frac{(a_2-a_1)k_2-1}{a_2}, - k_2, \frac{a_1k_2+1}{a_2}; \mathbf{P(Q)};\boldsymbol{\mu}_{n-1}\right)\notag\\
&\qquad\times G\left(- k_2, \frac{a_1 k_2+1}{a_2}\right)\notag\\
& +\frac{1}{\Gamma (s_T)} \mathcal R _{\M, {\mathbf{\eps}}} (\s), 
\end{align}
where
$$
\mathcal R _{\M, {\mathbf{\eps}}} (\s)= \sum_{k_1=0}^{M_1}\frac{(-1)^{k_1}}{2\pi \ii } \mathcal R_{M_2+\eps_2, k_1}(\s)+ \mathcal V_{M_2+\eps_2}(\s) + \mathcal R_{M_1+\eps_1}(\s),$$
with  
\begin{align*}
\mathcal{R}_{M_2+\eps_2, k_1}&(\s) = \frac{1}{2\pi \ii} \int_{(-M_2-\eps_2)}\Gamma (s_T+k_1- z_2) \Gamma (z_2) \\
& \times\zeta_{n-1}^{\dc}(\s', s_T+k_1-z_2, -k_1, z_2; \mathbf{P(Q)};\boldsymbol{\mu}_{n-1}) \frac{G(-k_1, z_2)}{a_2z_2-a_1 k_1-1}\, \dd z_2;\\
\mathcal{V}_{M_2+\eps_2}&(\s)= \frac{1}{2\pi \ii a_1} \int_{(M_2+\eps_2)}\Gamma \left(s_T+ \frac{(a_2-a_1)z_2-1}{a_1}\right) \Gamma (z_2)  \Gamma\left(\frac{1-a_2z_2}{a_1}\right)\\
 & \times\zeta_{n-1}^{\dc}\left(\s', s_T+ \frac{(a_2-a_1)z_2-1}{a_1}, \frac{1-a_2z_2}{a_1}, z_2; \mathbf{P(Q)};\boldsymbol{\mu}_{n-1}\right)\\
 &\qquad\times G(\frac{1-a_2z_2}{a_1}, z_2)\, \dd z_2;\\
\mathcal{R}_{M_1+\eps_1}&(\s) =  \frac{1}{(2\pi \ii)^2} \int_{(-M_1-\eps_1)}\int_{(\rho_2)}\Gamma (s_T-z_1-z_2) \Gamma (z_1) \Gamma (z_2) \\
& \times \zeta_{n-1}^{\dc}\left(\s', s_T-z_1-z_2, z_1, z_2; \mathbf{P(Q)};\boldsymbol{\mu}_{n-1}\right) \frac{G(z_1, z_2)}{a_1z_1+a_2z_2-1} \, \dd z_1 \, \dd z_2.
\end{align*}
It is clear that 

\begin{itemize}
\item[$\bullet$] $\s \mapsto \mathcal R_{M_2+\eps_2,k_1}(\s)$ is holomorphic in $\C^{T-1}\times \{\Re s_T > -M_2-\eps_2 -k_1\}$;
\item[$\bullet$] $\s \mapsto \mathcal V_{M_2+\eps_2}(\s)$ is holomorphic in $\C^{T-1}\times \left\{\Re s_T > \frac{1-(a_2-a_1) (M_2+\eps_2)}{a_1}\right\};$
\item[$\bullet$] $\s\mapsto \mathcal R_{M_1+\eps_1}(\s)$ is holomorphic in $\C^{T-1}\times \left\{\Re s_T > \rho_2-\eps_1-M_1\right\}$.
\end{itemize}

Let $K\in \N_0$. Choose $M_1 = K$ and $M_2= \max \left(K, \left\lceil\frac{1+a_1 K}{a_2-a_1}\right\rceil\right)$.
Under these choices, it follows that there exists $\delta >0$ such that $\s\mapsto \mathcal R _{\M, {\mathbf{\eps}}} (\s)$ has {\it holomorphic} continuation to $\C^{T-1} \times \{\Re s_T > -K-\delta\}$.
We deduce then from formula (\ref{keyformula1}) that $\s \mapsto \zeta_n(\s;\P;\mub_{n-1})$ has a meromorphic continuation 
to $\C^{T-1} \times \{\Re s_T > -K-\delta\}$ and that its possible singularities in this domain come only from the factors 
$$
\frac{\Gamma \left(s_T+\frac{(a_2-a_1)k_2-1}{a_2}\right)}{\Gamma (s_T)}\quad (k_2 \in \N_0, \, 0\leq k_2\leq  \left\lfloor\frac{a_2 M_2-1}{a_1}\right\rfloor).
$$
Whether these are really singular or not can be determined as
follows.

\begin{itemize}
\item[$\bullet$] If $a_2 \mid (a_2-a_1)k_2-1$ , that is $(a_2-a_1)k_2-1= a_2 u_2$ $(u_2 \in \N)$, then
$$\frac{\Gamma \left(s_T+\frac{(a_2-a_1)k_2-1}{a_2}\right)}{\Gamma (s_T)}=\frac{\Gamma\left(s_T+u_2\right)}{\Gamma (s_T)}$$ is holomorphic in $s_T\in \C$. 
\item[$\bullet$] If $a_2 \nmid (a_2-a_1)k_2-1$, that is $(a_2-a_1)k_2-1= a_2 (u_2-1) +r$ $(u_2, r \in \N_0, \, 1\leq r\leq a_2-1)$, then
$$\frac{\Gamma \left(s_T+\frac{(a_2-a_1)k_2-1}{a_2}\right)}{\Gamma (s_T)}=\frac{\Gamma\left(s_T+u_2-1+\frac{r}{a_2}\right)}{\Gamma (s_T)}$$
is meromorphic in $\C$ with only simple poles in the point $s_T=1-\frac{r}{a_2}-u_2-\ell$ $(\ell \in \N_0)$.
\end{itemize}

By letting $K$ to infinity, we deduce that
$\s \mapsto \zeta_n(\s;\P;\mub_{n-1})$ has a meromorphic continuation to the whole space $\C^T$ and that its singularities are located in the union of the hyperplanes $s_T=1-\frac{r}{a_2}-j$ $(j, r\in \N_0, \,1\leq r \leq a_2 -1)$.

Let $\Nb=(\Nb', N_T)\in \N_0^{T-1}\times \N_0$.
 Set $K=N_T$, $M_1 = K$ and $M_2= \max \left(K, \left\lceil\frac{1+a_1 K}{a_2-a_1}\right\rceil\right)$.
Since $ R _{\M, {\mathbf{\eps}}} (\s)$ is regular in $\s=-\Nb$, the functions  $\ds \Gamma \left(s_T+\frac{(a_2-a_1)k_2-1}{a_2}\right) $ ($a_2 \nmid (a_2-a_1)k_2-1$) are regular in $s_T=-N_T$ and $\displaystyle \frac{1}{\Gamma (s_T)}\Big|_{s_T=-N_T}=0$,
we deduce from (\ref{keyformula1}) that
\begin{align}\label{keyformula2}
&\zeta_n(-\Nb;\P;\mub_{n-1})\notag\\
&=-\sum_{k_1=0}^{M_1}\sum_{k_2=0}^{M_2}\frac{(-1)^{k_1+k_2}}{k_1! k_2!}
\frac{\Gamma (s_T+k_1+k_2)}{\Gamma (s_T)}\Big|_{s_T=-N_T} \notag \\
& \times\zeta_{n-1}^{\dc}(-\Nb', -N_T+k_1+k_2, -k_1, -k_2; \mathbf{P(Q)};\boldsymbol{\mu}_{n-1})\frac{G(-k_1, -k_2)}{a_2k_2+a_1k_1+1}\notag\\
 & + \sum_{\substack{k_2=1 \\ a_2 \mid a_1 k_2+1}}^{\left\lfloor\frac{a_2 M_2-1}{a_1}\right\rfloor}\frac{(-1)^{k_2}
\left(\frac{a_1 k_2+1}{a_2}-1\right)!}{k_2! a_2} \,
\frac{\Gamma \left(s_T+\frac{(a_2-a_1)k_2-1}{a_2}\right)}{\Gamma (s_T)}\Big|_{s_T=-N_T} \notag 
\\
&\times \zeta_{n-1}^{\dc}\left(-\Nb', -N_T+\frac{(a_2-a_1)k_2-1}{a_2}, - k_2, \frac{a_1k_2+1}{a_2}; \mathbf{P(Q)};\boldsymbol{\mu}_{n-1}\right)\notag\\
&\qquad \times G\left(- k_2, \frac{a_1 k_2+1}{a_2}\right).
\end{align}
Moreover, it is easy to see that

\begin{itemize}
\item[$\bullet$] $\displaystyle \frac{\Gamma (s_T+k)}{\Gamma (s_T)}\Big|_{s_T=-N_T}= \prod_{j=0}^{k-1}(-N_T+j) :=(-N_T)_{k}$\\ where $(x)_n:=\prod_{j=0}^{n-1} (x+j)$ is the Pochhammer symbol,
hence especially the above is equal to $0$ if $k>N_T$;
\item[$\bullet$] $\displaystyle G(-k_1, -k_2) = -(a_1k_1+a_2k_2+1) \zeta \left(-(a_1k_1+a_2k_2)\right)$, 
which is equal to
$(-1)^{a_1k_1+a_2k_2+1} B_{a_1k_1+a_2k_2+1}$;
\item[$\bullet$] $G\left(- k_2, \frac{a_1 k_2+1}{a_2}\right)=1$; 
\item[$\bullet$] $\displaystyle \frac{\Gamma (s_T+k_1+k_2)}{\Gamma (s_T)}\Big|_{s_T=-N_T}= 0$ if $k_1+k_2 \geq N_T+1$;
\item[$\bullet$] $\displaystyle \frac{\Gamma (s_T+k_1+k_2)}{\Gamma (s_T)}\Big|_{s_T=-N_T}= \frac{(-1)^{k_1+k_2}N_T!}{(N_T-k_1-k_2)!}$ if $k_1+k_2 \leq N_T$.
\end{itemize}
Substituting these five points into formula (\ref{keyformula2}), we obtain:
\begin{thm}\label{thm_d=2}
When $P_T$ is of the form \eqref{P_t_form2}, we have for $\Nb \in \N_0^T$, 
\begin{align}\label{keyformula3}
&\zeta_n(-\Nb;\P;\mub_{n-1})\notag\\
&=\sum_{\substack{k_1, k_2 \geq 0 \\ k_1+k_2 \leq N_T}}\frac{N_T !}{(N_T-k_1-k_2)! k_1! k_2!}  \notag \\
&\times\zeta_{n-1}^{\dc}(-\Nb', -N_T+k_1+k_2, -k_1, -k_2; \mathbf{P(Q)};\boldsymbol{\mu}_{n-1})\notag\\
&\qquad\times\frac{(-1)^{a_1k_1+a_2k_2} B_{a_1k_1+a_2k_2+1}}{a_2k_2+a_1k_1+1}
\notag\\
&  + \sum_{\substack{k_2=1 \\ a_2  | a_1 k_2+1}}^{\left\lfloor\frac{a_2 N_T+1}{a_2-a_1}\right\rfloor}\frac{(-1)^{k_2}
\left(\frac{a_1 k_2+1}{a_2}-1\right)!}{k_2! a_2}
\left(-N_T\right)_{\frac{(a_2-a_1)k_2-1}{a_2}} \notag 
\\
& \times\zeta_{n-1}^{\dc}\left(-\Nb', -N_T+\frac{(a_2-a_1)k_2-1}{a_2}, - k_2, \frac{a_1k_2+1}{a_2}; \mathbf{P(Q)};\boldsymbol{\mu}_{n-1}\right).
\end{align}
\end{thm}

\begin{rem}
Using formula (\ref{keyformula3}), we obtain an explicit expression of the values of
$\zeta_n(-\Nb;\P;\mub_{n-1})$ at non-negative integer points in terms of the values of de Crisenoy's series $\zeta_{n-1}^{\dc}(\cdot;\mathbf{P(Q)};\boldsymbol{\mu}_{n-1})$ in rational points. 
In the first sum on the right-hand side of formula (\ref{keyformula3}) only the values of $\zeta_{n-1}^{\dc}(\cdot; \mathbf{P(Q)};\boldsymbol{\mu}_{n-1})$ at non-negative integer points appear. So thanks to de Crisenoy's work this sum is an element of the field generated over $\Q$ by $\mu_1,\ldots, \mu_{n-1}$ and the coefficients of the polynomials $P_1,\ldots, P_T$.
However, in the second sum there appear factors of the form 
$$
\zeta_{n-1}^{\dc}\left(-\Nb', -N_T+\frac{(a_2-a_1)k_2-1}{a_2}, - k_2, \frac{a_1k_2+1}{a_2}; \mathbf{P(Q)};\boldsymbol{\mu}_{n-1}\right).
$$ 
We observe that
these factors are not necessary in the field generated over $\Q$ by $\mu_1,\ldots,\mu_{n-1}$ and the coefficients of the polynomials $P_1,\ldots, P_T$, and can be {\it highly transcendental} (see Example \ref{transcendental} in Section 8).
\end{rem}

\begin{example}
Assume that $n=d=T=2$, $a_1=1$ and $a_2=2$. Then $k_2$ in the second sum on the right-hand side of \eqref{keyformula3} should be odd, so
we put $k_2=2u_2+1$.
Then \eqref{keyformula3} implies that for $\Nb=(N_1, N_2)\in \N_0^2$ we have
\begin{align}\label{keyformula3bis}
&\zeta_2(-\Nb,\P,\mu_1)=\sum_{\substack{k_1, k_2 \geq 0 \\ k_1+k_2 \leq N_2}}\frac{N_2 !}{(N_2-k_1-k_2)! k_1! k_2!} \notag \\
& \qquad\times\zeta_1^{\dc}(-N_1, -N_2+k_1+k_2, -k_1, -k_2; \mathbf{P(Q)},\mu_1) \frac{(-1)^{k_1} B_{k_1+2k_2+1}}{k_1+2k_2+1} \notag\\
 & - \sum_{u_2=0}^{N_2}\frac{(-1)^{u_2} (u_2!)^2 \binom{N_2}{u_2}}{2(2u_2+1)!} \zeta_1^{\dc}(-N_1, -N_2+u_2, -2u_2-1, u_2+1; \mathbf{P(Q)},\mu_1).
\end{align}
\end{example}

\section{General case}
\label{section:general_case}

Now we proceed to the discussion of the general case. In this section we state our main general results. Recall that $\boldsymbol{\mu}_{n-1}=(\mu_1,\ldots,\mu_{n-1}) \in (\mathbb{T}\setminus\{1\})^{n-1}$.   Consider polynomials 
$$ P_1,\ldots,P_{T-1} \in \mathbb{R}[X_1,\ldots,X_{n-1}] $$ 
satisfying HDF, with $T \geq 1, \ n \geq 1$, and one more polynomial $P_T=P_T(X_1,\ldots,X_n)$ defined by \eqref{1-5} with $a_0<\cdots<a_d$ being integers, where $Q_0,\ldots,Q_d \in \mathbb{R}[X_1,\ldots,X_{n-1}]$ are polynomials satisfying HDF. Furthermore, if $n \geq 2$, we assume that \eqref{eq:condition_polynomials_tends_infinity} holds. We set $\alpha_j:=a_j-a_0$ $(1 \leq j \leq d)$. Observe that the tuple $\boldsymbol{\alpha}:=(\alpha_1,\ldots,\alpha_d)\in \mathbb{N}^d$ is a strictly increasing sequence of positive integers (i.e. $0 < \alpha_1 < \alpha_2 < \cdots < \alpha_d$).

In this section, we give the meromorphic continuation of $\zeta_{n}(\mathbf{s};\mathbf{P};\boldsymbol{\mu}_{n-1})$ to the whole complex space $\C^T$ and describe the set of its singularities in Theorem \ref{th:analytic_continuation_zeta_general_case} and Corollary \ref{cor:set_of_singularities_zeta_n_general_case}. We study its values at non-positive integer points in Theorem \ref{th:value_zeta_n}.

Note that in the case $d=0$, when $P_T(X_1,\ldots,X_n)=Q(X_1,\ldots,X_{n-1})X_n^{a_0}$, the function $\zeta_{n}(\mathbf{s};\mathbf{P};\boldsymbol{\mu}_{n-1})$ corresponds to a product of the Riemann zeta-function and a completely twisted multiple zeta-function, thus its meromorphic continuation and values can be directly studied by well-known facts about the Riemann zeta-function and by de Crisenoy's results \cite{dC}. Therefore, we now assume that $d>0$.

\subsection{Notations}

Let $\mathbf{k}=(k_1,\ldots,k_r) \in \mathbb{N}_0^r, \ \mathbf{x}=(x_1,\ldots,x_r), \ \mathbf{y}=(y_1,\ldots,y_r) \in \mathbb{C}^r$ where $r \in \mathbb{N}_0$. We write
$$ \mathbf{k}!:=k_1! \cdots k_r!, \quad \langle \boldsymbol{x}|\mathbf{y} \rangle:= x_1 y_1+\cdots +x_r y_r, \quad |\mathbf{x}|:=x_1+\cdots+x_r. $$
When $r=0$, the only tuple in $\mathbb{N}^0_0$ is the empty tuple, denoted by $\emptyset$. We assume that
$$ \emptyset !=1, \quad \langle \emptyset | \emptyset \rangle=0, \quad |\emptyset|=0. $$

For all $x,y \in \mathbb{R}\cup \{ -\infty,+\infty \}$, we define
$$ \mathcal{H}_T(x,y):= \{ \mathbf{s} \in \mathbb{C}^T \mid y>\sigma_T>x \}. $$
For all $\varepsilon \in (0,1)$, $M_1,\ldots,M_d \in \mathbb{N}_0$, $\eta_1,\ldots,\eta_d \in (0,1)$ we set
$$ \mathcal{H}_{\varepsilon}(\mathbf{M},\boldsymbol{\eta}) := \mathcal{H}_T\left(\max\left(\varepsilon,\frac{1+\varepsilon}{a_d}\right)-\frac{x(\mathbf{M},\boldsymbol{\eta})}{a_d},\frac{1-\varepsilon+y(\mathbf{M},\boldsymbol{\eta})}{a_0}\right), $$
where 
\begin{align*}
    x(\mathbf{M},\boldsymbol{\eta}):=&\min_{1 \leq j \leq d-1}\{(a_d-a_j)(M_j+\eta_j),M_d+\eta_d\}, \\
    y(\mathbf{M},\boldsymbol{\eta}):=&\min_{1 \leq j \leq d}\{(a_j-a_0)(M_j+\eta_j)\}.
\end{align*}
If $a_0=0$, we replace $\frac{1-\varepsilon+y(\mathbf{M},\boldsymbol{\eta})}{a_0}$ with $+\infty$ in the definition of $\mathcal{H}_{\varepsilon}(\mathbf{M},\boldsymbol{\eta})$.

\subsection{Statement of theorems}

\begin{definition}
\label{def:F}
Let $\mathbf{P(Q)}:=(P_1,\ldots,P_{T-1},Q_0,\ldots, Q_d)$. We define the function $F:\mathbb{C}^{T+d} \to \mathbb{C}$ as the holomorphic continuation to the whole complex space of
\begin{align*}
    F(\mathbf{s};\mathbf{z}):= (a_0s_T+\langle \boldsymbol{\alpha},\mathbf{z} \rangle-1) &\zeta(a_0s_T+\langle \boldsymbol{\alpha},\mathbf{z} \rangle) \\
    & \times \zeta_{n-1}^{\dc}\left(s_1,\ldots,s_{T-1},s_T-|\mathbf{z}|,\mathbf{z}; \mathbf{P(Q)};\boldsymbol{\mu}_{n-1}\right),
\end{align*}
where $\mathbf{s}=(s_1,\ldots,s_T) \in \mathbb{C}^T$, $\mathbf{z}=(z_1,\ldots,z_d) \in \mathbb{C}^d$, and where $\zeta^{\dc}$ is defined in \eqref{eq:def_zeta_de_crisenoy}.
\end{definition}

First observe that by Proposition \ref{thm_dC_compl}, $F$ is of moderate growth. Also note that in the case $n=1$, the polynomials $P_1,\ldots,P_{T-1},Q_0,\ldots, Q_d$ correspond to real numbers $p_1,\ldots,p_{T-1},q_0,\ldots,q_d \in \mathbb{R} \setminus\{0\}$, and by convention we have
$$ \zeta_{0}(s_1,\ldots,s_{T-1},s_T-|\mathbf{z}|,\mathbf{z}; \mathbf{P(Q)};\boldsymbol{\mu}_0)=q_0^{-s_T} \left(\prod_{t=1}^{T-1} p_t^{-s_t}\right) \prod_{j=1}^d \left(\frac{q_0}{q_j}\right)^{z_j}. $$

\begin{thm}
\label{th:analytic_continuation_zeta_general_case}
Let $\epsilon>0$ be sufficiently small. Then $\zeta_{n}(\mathbf{s};\mathbf{P};\boldsymbol{\mu}_{n-1})$ has a meromorphic continuation on $\mathcal{H}_{\varepsilon}(\mathbf{M},\boldsymbol{\eta})$ by the form given by
\begin{align}
    \zeta_{n}&(\mathbf{s};\mathbf{P};\boldsymbol{\mu}_{n-1})= \frac{H(\mathbf{s})}{\Gamma(s_T)}+ \sum_{k_1=0}^{M_1} \cdots \sum_{k_d=0}^{M_d} \frac{(-1)^{|\mathbf{k}|}}{\mathbf{k}!} \frac{F(\mathbf{s};-\mathbf{k})\Gamma(s_T+|\mathbf{k}|)}{\Gamma(s_T)(a_0s_T-\langle \boldsymbol{\alpha}|\mathbf{k}\rangle-1)} \label{eq:analytic_continuation_zeta_general_case} \\ 
    & +\sum_{\ell_1=0}^{M_1} \cdots \sum_{\ell_{d-1}=0}^{M_{d-1}} \frac{(-1)^{|\boldsymbol{\ell}|}}{(a_d-a_0)\boldsymbol{\ell}!} \frac{F(\mathbf{s};(-\boldsymbol{\ell},g(s_T;-\boldsymbol{\ell})))\Gamma(\widetilde{g}(s_T;-\boldsymbol{\ell}))\Gamma(g(s_T;-\boldsymbol{\ell}))}{\Gamma(s_T)} \nonumber
\end{align}
with $\mathbf{k}=(k_1,\ldots,k_d)$ and
$\boldsymbol{\ell}=(\ell_1,\ldots,\ell_{d-1})$,
where $H(\mathbf{s})$ is a holomorphic function on $\mathcal{H}_{\varepsilon}(\mathbf{M},\boldsymbol{\eta})$, $F$ is defined in Definition \ref{def:F}, and we set for all $s_T \in \mathbb{C}, \ \mathbf{z}=(z_1,\ldots,z_{d-1}) \in \mathbb{C}^{d-1}$,
\begin{align*}
    g(s_T;\mathbf{z}):=& \frac{1-a_0s_T-\sum_{j=1}^{d-1} (a_j-a_0)z_j}{a_d-a_0}, \\
    \widetilde{g}(s_T;\mathbf{z}):=& s_T-g(s_T;\mathbf{z})-|\mathbf{z}| = \frac{a_ds_T-\sum_{j=1}^{d-1} (a_d-a_j)z_j-1}{a_d-a_0}.
\end{align*}
\end{thm}

By convention, the last sum of \eqref{eq:analytic_continuation_zeta_general_case} is non-empty when $d=1$. The only tuple $\boldsymbol{\ell} \in \mathbb{N}_0^0$ corresponds to the empty tuple $\emptyset$. Recall that $\emptyset ! =1$, and $|\emptyset|=0$. Since we assume that sums over the empty set vanish, we also have $g(s_T;\emptyset)=\frac{1-a_0s_T}{a_1-a_0}$ and $\widetilde{g}(s_T;\emptyset)=\frac{a_1s_T-1}{a_1-a_0}$.

\begin{cor}
\label{cor:set_of_singularities_zeta_n_general_case}
The function $\zeta_{n}(\mathbf{s};\mathbf{P};\boldsymbol{\mu}_{n-1})$ has a meromorphic continuation to the whole complex space $\mathbb{C}^T$, with singularities belonging to a countable union of hyperplanes defined by
\begin{align*}
    \mathcal{S}_{\boldsymbol{\ell}}^{(1)}:=& \{ \mathbf{s} \in \mathbb{C}^T \mid g(s_T;-\boldsymbol{\ell}) \in \mathbb{Z}_{\leq 0} \} && (\boldsymbol{\ell}=(\ell_1,\ldots,\ell_{d-1}) \in \mathbb{N}_0^{d-1}), \\
    \mathcal{S}_{\boldsymbol{\ell}}^{(2)}:=& \{ \mathbf{s} \in \mathbb{C}^T \mid \widetilde{g}(s_T;-\boldsymbol{\ell}) \in \mathbb{Z}_{\leq 0} \} && (\boldsymbol{\ell}=(\ell_1,\ldots,\ell_{d-1}) \in \mathbb{N}_0^{d-1}), \\
    \mathcal{S}_{\mathbf{k}}^{(3)}:=& \left\{ \mathbf{s} \in \mathbb{C}^T \mid a_0s_T=1+\langle \boldsymbol{\alpha}|\mathbf{k} \rangle \right\} && (\mathbf{k}=(k_1,\ldots,k_d) \in \mathbb{N}_0^d).
\end{align*}
Furthermore, non-positive integer points $\mathbf{s}=-\mathbf{N} \in \mathbb{Z}_{\leq 0}^T$ are not singularities.
\end{cor}

\begin{proof}
By Theorem \ref{th:analytic_continuation_zeta_general_case}, we know that $\zeta_{n}(\mathbf{s};\mathbf{P};\boldsymbol{\mu}_{n-1})$ has a meromorphic continuation to $\mathbb{C}^T$ by taking $M_1,\ldots,M_d$ sufficiently large. Also, we see that its singularities come exclusively from one of the following three terms:
$$ \Gamma(g(s_T;-\boldsymbol{\ell})), \quad \Gamma(\widetilde{g}(s_T;-\boldsymbol{\ell})), \quad \frac{1}{a_0s_T-\langle \boldsymbol{\alpha}|\mathbf{k} \rangle-1}. $$
The fact that non-positive integer points $-\mathbf{N} \in \mathbb{Z}_{\leq 0}^T$ are not singularities follows from the fact that $\frac{1}{\Gamma}$ vanishes at non-positive integers. For more details on this, we refer to the proof of Theorem \ref{th:value_zeta_n}, in Subsection \ref{subsection:values_zeta_n_general_case}.
\end{proof}

By taking $M_1,\ldots,M_d$ sufficiently large in Theorem \ref{th:analytic_continuation_zeta_general_case}, and by studying the various limits involving gamma functions appearing on the right-hand side of \eqref{eq:analytic_continuation_zeta_general_case}, we obtain the following explicit formula for the values at non-positive integer points of $\zeta_{n}(\mathbf{s};\mathbf{P};\boldsymbol{\mu}_{n-1})$ using only special values of fully twisted multiple zeta-functions $\zeta_{n-1}^{\dc}(\cdot; \cdot; \boldsymbol{\mu}_{n-1})$, twisted Bernoulli numbers $\widetilde{B}_n:= (-1)^{n}B_n$, and multinomial coefficients.

\begin{thm}
\label{th:value_zeta_n}
Let $-\mathbf{N}=(-N_1,\ldots,-N_T) \in \mathbb{Z}_{\leq 0}^T$, then
\begin{align}
    \zeta_{n}(-&\mathbf{N};\mathbf{P};\boldsymbol{\mu}_{n-1})=-\sum_{\substack{k_1,\ldots,k_d \geq 0 \\ |\mathbf{k}| \leq N_T}} \binom{N_T}{(\mathbf{k},N_T-|\mathbf{k}|)} \frac{\widetilde{B}_{a_0N_T+\langle \boldsymbol{\alpha}|\mathbf{k}\rangle+1}}{a_0N_T+\langle \boldsymbol{\alpha}|\mathbf{k} \rangle+1} \label{eq:values_zeta_n} \\
    & \hspace{89pt} \times \zeta_{n-1}^{\dc}(-N_1,\ldots,-N_{T-1},-N_T+|\mathbf{k}|,-\mathbf{k};\mathbf{P(Q)};\boldsymbol{\mu}_{n-1}) \nonumber \\
    &+ \frac{N_T!}{a_d}\sum_{i=0}^{\left\lfloor \frac{a_dN_T+1}{\alpha_d}\right\rfloor} \sum_{\substack{\ell_1,\ldots,\ell_{d-1} \geq 0 \\ \sum_{j=1}^{d-1} (a_d-a_j) \ell_j=1+a_dN_T-i\alpha_d}} \frac{(-1)^{|\boldsymbol{\ell}|+i+N_T}(|\boldsymbol{\ell}|+i-N_T-1)!}{i!\boldsymbol{\ell}!} \nonumber \\
    & \hspace{60pt} \times \zeta_{n-1}^{\dc}(-N_1,\ldots,-N_{T-1},-i,-\boldsymbol{\ell},|\boldsymbol{\ell}|+i-N_T;\mathbf{P(Q)};\boldsymbol{\mu}_{n-1}). \nonumber
\end{align}
\end{thm}

Observe that when $d=1$ the tuple $\boldsymbol{\ell}$ in the second sum of the right-hand side of \eqref{eq:values_zeta_n} corresponds to the empty tuple $\emptyset$, and the last sum corresponds to summing over all integers $0 \leq i \leq \lfloor \frac{a_1N_T+1}{\alpha_1}\rfloor$ such that $i\alpha_1=1+a_1 N_T$. Therefore, the last sum of \eqref{eq:values_zeta_n} does not vanish if and only if $\alpha_1\mid (a_1 N_t +1)$. In such a case, the only integer $i$ satisfying the condition in the sum is $i=\frac{a_1N_T+1}{\alpha_1}$.

\section{Proofs of the theorems for the general case}\label{section:proofs}


In this section we prove the results stated in the preceding section.

In Subsection \ref{subsection:integral_representation_formula}, we prove that $\zeta_{n}(\mathbf{s};\mathbf{P};\boldsymbol{\mu}_{n-1})$ has holomorphic continuation to a suitable domain, and via a Mellin transform we show an integral expression of it.

In Subsection \ref{subsection:recursive_structure}, we introduce two types of integrals, $I^{(r)}$ and $J^{(r)}$, that both appear when we iterate to apply recursively the residue theorem on the integral expression formula for $\zeta_{n}(\mathbf{s};\mathbf{P};\boldsymbol{\mu}_{n-1})$. In Proposition \ref{prop:formula_zeta_n_shift_left}, we first show a relation between $\zeta_{n}(\mathbf{s};\mathbf{P};\boldsymbol{\mu}_{n-1})$ and $I^{(d-1)}$ and $J^{(d-1)}$, and a holomorphic term on a suitable domain. In Proposition \ref{prop:iterative_structure_I_r} and Proposition \ref{prop:iterative_structure_J_r}, we further apply the residue theorem to both integrals $I^{(r)}$ and $J^{(r)}$, and we express them in terms of $I^{(r-1)}$ and $J^{(r-1)}$, plus a holomorphic function on a suitable domain.

In Subsection \ref{subsection:meromorphic_continuation_zeta_n}, we prove the meromorphic continuation of $\zeta_{n}(\mathbf{s};\mathbf{P};\boldsymbol{\mu}_{n-1})$. We proceed by first applying Proposition \ref{prop:formula_zeta_n_shift_left}, then by recursively applying Propositions \ref{prop:iterative_structure_I_r} and \ref{prop:iterative_structure_J_r}. When doing that, we first have to restrict ourselves to the set $\mathcal{H}_{\varepsilon}(\mathbf{M},\boldsymbol{\eta}) \subset \mathbb{C}^T$ as defined in Theorem \ref{th:analytic_continuation_zeta_general_case}, because it contains all the domains of holomorphy of all the intermediate holomorphic functions in Propositions \ref{prop:formula_zeta_n_shift_left}, \ref{prop:iterative_structure_I_r} and \ref{prop:iterative_structure_J_r}. By the induction principle, we will show that $\zeta_{n}(\mathbf{s};\mathbf{P};\boldsymbol{\mu}_{n-1})$ can be expressed as a sum of meromorphic functions defined on a sufficiently large set.

In Subsection \ref{subsection:values_zeta_n_general_case}, we study the values at non-positive integer points of $\zeta_{n}(\mathbf{s};\mathbf{P};\boldsymbol{\mu}_{n-1})$ by evaluating at non-positive integer points the formula obtained upon the study of the meromorphic continuation of $\zeta_{n}(\mathbf{s};\mathbf{P};\boldsymbol{\mu}_{n-1})$ in Theorem \ref{th:analytic_continuation_zeta_general_case}.

\subsection{Integral expression formula}
\label{subsection:integral_representation_formula}

\begin{prop}
 For all $\rho_1, \ldots, \rho_d>0$, the function $\zeta_{n}(\mathbf{s};\mathbf{P};\boldsymbol{\mu}_{n-1})$ has a holomorphic continuation to the set
$$ \mathcal{D}_{\boldsymbol{\rho}}(\mathbf{P}):=\left\{\mathbf{s} \in \mathbb{C}^T \mid \sigma_T>\rho_1+\cdots +\rho_d \text{ and }a_0\sigma_T+\alpha_1\rho_1+\cdots +\alpha_d\rho_d>1 \right\} $$
given by
\begin{equation}
    \zeta_{n}(\mathbf{s};\mathbf{P};\boldsymbol{\mu}_{n-1}) = \frac{1}{(2 \pi \ii)^d} \int_{(\rho_1)} \cdots \int_{(\rho_d)} \frac{\Gamma(s_T-|\mathbf{z}|)F(\mathbf{s};\mathbf{z}) \prod_{j=1}^d \Gamma(z_j) \, \dd z_j }{\Gamma(s_T)(a_0s_T+\langle \boldsymbol{\alpha}|\mathbf{z} \rangle-1)}  \label{eq:integral_representation_zeta_n}
    \end{equation}
where $F(\mathbf{s};\mathbf{z})$ is defined in Definition \ref{def:F}.
\end{prop}

\begin{proof}
Recall that by Proposition \ref{conv_domain}, $\zeta_n(\mathbf{s}, \mathbf{P}, \mu_{n-1})$ absolutely converges in the domain $\Re s_1,\ldots,\Re s_{T-1}> 1/b$ and $\Re s_T> 1/a$. In order to prove the integral expression formula \eqref{eq:integral_representation_zeta_n}, we first consider $\rho'_1>\rho_1,\ldots,\rho'_d>\rho_d$ sufficiently large, and $\Re s_1,\ldots,\Re s_T>\max ( \frac{1}{a}, \frac{1}{b})$ sufficiently large so that the Dirichlet series
$$ \sum_{m_1,\ldots,m_{n-1} \geq 1} \frac{\mu_1^{m_1}\cdots \mu_{n-1}^{m_{n-1}}}{Q_0(\mathbf{m})^{s_T-|\mathbf{z}|}\prod_{j=1}^{T-1}P_j(\mathbf{m})^{s_j}\prod_{i=1}^d Q_i(\mathbf{m})^{z_i}} $$
absolutely converges for all $z_1 \in (\rho'_1),\ldots, z_d \in (\rho'_d)$. By the multiple Mellin-Barnes formula
(Proposition \ref{MMB_formula}), we obtain
\begin{align*}
    \zeta_{n}(&\mathbf{s};\mathbf{P};\boldsymbol{\mu}_{n-1}) \\
    =& \sum_{m_1,\ldots,m_n \geq 1} \frac{\mu_1^{m_1} \cdots \mu_{n-1}^{m_{n-1}}}{\left(Q_0(\mathbf{m})m_n^{a_0}\left(1+\sum_{j=1}^d \frac{Q_j(\mathbf{m})}{Q_0(\mathbf{m})}m_n^{a_j-a_0}\right)\right)^{s_T} \prod_{i=1}^{T-1} P_i(\mathbf{m})^{s_i}} \\
    =& \frac{1}{(2\pi \ii)^d}\sum_{m_1,\ldots,m_n \geq 1} \int_{(\rho'_1)} \cdots \int_{(\rho'_d)} \frac{\mu_1^{m_1}\cdots \mu_{n-1}^{m_{n-1}}}{Q_0(\mathbf{m})^{s_T-|\mathbf{z}|} \left(\prod_{i=1}^{T-1} P_i(\mathbf{m})^{s_i}\right) \prod_{j=1}^d Q_j(\mathbf{m})^{z_j}} \\
    & \hspace{150pt} \times \frac{1}{m_n^{a_0s_T+\langle \boldsymbol{\alpha}|\mathbf{z} \rangle}} \times \frac{\Gamma(s_T-|\mathbf{z}|)\prod_{j=1}^d \Gamma(z_j) \, \dd z_j}{\Gamma(s_T)}.
\end{align*}
By our assumption on $\mathbf{s}$ and $\boldsymbol{\rho}'$, we can perform a series-integral inversion in order to obtain the formula \eqref{eq:integral_representation_zeta_n} (with $\boldsymbol{\rho}'$ instead of $\boldsymbol{\rho}$). We then have
\begin{equation}
    \zeta_{n}(\mathbf{s};\mathbf{P};\boldsymbol{\mu}_{n-1}) = \frac{1}{(2 \pi \ii)^d} \int_{(\rho'_1)} \cdots \int_{(\rho'_d)} \frac{\Gamma(s_T-|\mathbf{z}|)F(\mathbf{s};\mathbf{z}) \prod_{j=1}^d \Gamma(z_j) \, \dd z_j }{\Gamma(s_T)(a_0s_T+\langle \boldsymbol{\alpha}|\mathbf{z} \rangle-1)}. \label{eq:integral_representation_rho_prime}
\end{equation}
We know by Proposition \ref{thm_dC_compl} that $F$ is an entire function of moderate growth, thus it follows that the right-hand side of \eqref{eq:integral_representation_zeta_n} is holomorphic on the domain $\mathcal{D}_{\boldsymbol{\rho}'}(\mathbf{P})$. Therefore, \eqref{eq:integral_representation_rho_prime} gives an analytic continuation of $\zeta_{n}(\mathbf{s};\mathbf{P};\boldsymbol{\mu}_{n-1})$ to the domain $\mathcal{D}_{\boldsymbol{\rho}'}(\mathbf{P})$.

It remains to perform a shift to the left, from $(\rho'_1) \text{ to } (\rho_1), \ldots, (\rho'_d) \text{ to } (\rho_d)$ in \eqref{eq:integral_representation_rho_prime}. Let $\mathbf{s} \in \mathcal{D}_{\boldsymbol{\rho}'}(\mathbf{P}) \subset \mathcal{D}_{\boldsymbol{\rho}}(\mathbf{P})$. In the integrand of \eqref{eq:integral_representation_rho_prime}, there is no singularity in the variables $z_1,\ldots,z_d$ when $\rho_1 \leq \Re z_1 \leq \rho'_1, \ldots, \rho_d \leq \Re z_d \leq \rho'_d$. Therefore, shifting \eqref{eq:integral_representation_rho_prime} to the left with respect to $z_1,\ldots,z_d$ introduces no singularities, and applying the residue theorem proves that \eqref{eq:integral_representation_zeta_n} holds for all $\mathbf{s} \in \mathcal{D}_{\boldsymbol{\rho}'}(\mathbf{P})$. We can then extend this result for all $\mathbf{s} \in \mathcal{D}_{\boldsymbol{\rho}}(\mathbf{P})$ by analytic continuation.
\end{proof}

It is clear that the domain $\mathcal{D}_{\boldsymbol{\rho}}(\mathbf{P})$ does not contain any non-positive integer points $-\mathbf{N} \in \mathbb{Z}_{\leq 0}^n$, which is why we will perform many shifts to the left in the integral at the right-hand side of \eqref{eq:integral_representation_zeta_n} in order to obtain a meromorphic continuation to the whole complex space. Such a shift is possible because the gamma-factor has exponential decay along vertical strips, and as mentioned in the proof above, the function $F(\mathbf{s};\mathbf{z})$ also has moderate growth. For some technical reasons, we choose to perform the first shift to the left with respect to the variable $z_d$. Note that two types of poles will contribute in the residue theorem, one coming from 
$\displaystyle{\frac{1}{a_0s_T+\langle \boldsymbol{\alpha}|\mathbf{z} \rangle-1}}$, and several poles coming from $\Gamma(z_d)$. 

If we iterate to apply the residue theorem, we then see some auxiliary integrals appearing in our computations. At each step, we expect to have two different types of poles to deal with, both of which look like the ones we described in the previous paragraph. Fortunately, by assuming some conditions on the real part of $s_T$, we will have one fewer pole contributing to the residue theorem.

\subsection{Auxiliary integrals}
\label{subsection:recursive_structure}

In the rest of this section, we assume the following:
\begin{assu}
\label{assumption:delta_epsilon_rho}
Let $\delta>\varepsilon>0$ be real numbers such that
\begin{equation}
    \delta a_0<1-2\varepsilon<1+(\alpha_d+1)\varepsilon<\delta a_1. \label{eq:ineq_delta_a0_delta_a1}
\end{equation}
Let $\rho_1,\ldots,\rho_d>0$ be positive numbers such that
\begin{align}
    &\alpha_1\rho_1+\cdots +\alpha_{d-1} \rho_{d-1}<\varepsilon \label{eq:ineq_alpha_rho_epsilon}\\
    &(\alpha_d-\alpha_1)\rho_1+\cdots +(\alpha_d-\alpha_{d-1}) \rho_{d-1}<\varepsilon \label{eq:ineq_alpha_rho_epsilon_bis}\\
    &\alpha_1\rho_1+\cdots +\alpha_d\rho_d+a_0\delta>1 \label{eq:ineq_delta_sum_alpha_rho} \\
    &\delta>\rho_1+\cdots +\rho_d. \label{eq:ineq_delta_sum_rho}
\end{align}
\end{assu}

\begin{example}
We can take the following explicit $\delta, \varepsilon, \rho_1,\ldots,\rho_d$ that verify the previous assumption:
\begin{enumerate}[label=(\roman*)]
    \item If $a_0>0$, then we consider
    $$\delta:=\frac{1}{2a_0}+\frac{1}{2a_1}, \quad \varepsilon \leq \frac{1}{8a_0(a_d+1)}. $$
    \item If $a_0=0$, then we consider
    $$ \delta:=\frac{3}{2}, \quad \varepsilon \leq \frac{1}{4(a_d+1)}. $$
    \item For the positive numbers $\rho_1,\ldots,\rho_d$, we consider 
    $$ \rho_d = \frac{1-a_0\delta+\varepsilon}{\alpha_d}, \qquad \rho_j \leq \frac{\varepsilon}{d\alpha_d} \quad (1 \leq j \leq d-1). $$
\end{enumerate}
\end{example}

In this subsection, we aim to use iteratively the residue theorem on the integral at the right-hand side of \eqref{eq:integral_representation_zeta_n}, with respect to the last variable of $\mathbf{z}$ at every step. By doing so, we will encounter auxiliary integrals that look like the one in \eqref{eq:integral_representation_zeta_n}. For the sake of simplicity, we now assume for this subsection that $\mathbf{s} \in \mathcal{H}_T(\delta,\delta+\frac{\varepsilon}{a_0})$, i.e. 
$$ \delta < \Re s_T < \delta+\frac{\varepsilon}{a_0}. $$
By assuming this condition on $\mathbf{s}$, we will avoid dealing with undesirable singularities when applying the residue theorem to the two types of auxiliary integrals defined below.

\begin{definition}
Let $\mathbf{k}=(k_{r+1},\ldots,k_d) \in \mathbb{N}_0^{d-r}$ be a tuple where $0 \leq r \leq d-1$. For all $\mathbf{s} \in \mathcal{H}_T(\delta,\delta+\frac{\varepsilon}{a_0})$, we set
\begin{align}
    I^{(r)}_{\mathbf{k}}(\mathbf{s}):=&\frac{1}{(2\pi \ii)^r} \int_{(\rho_1)} \cdots \int_{(\rho_r)} \frac{\Gamma(s_T+|\mathbf{k}|-|\mathbf{z}|) F(\mathbf{s};\mathbf{z},-\mathbf{k}) \prod_{j=1}^r\Gamma(z_j) \, \dd z_j}{\Gamma(s_T)(a_0s_T+\langle \boldsymbol{\alpha}|(\mathbf{z},-\mathbf{k})\rangle-1)}, \label{eq:def_I_r} \\
    I^{(0)}_{\mathbf{k}}(\mathbf{s}):=& \frac{F(\mathbf{s};-\mathbf{k})\Gamma(s_T+|\mathbf{k}|)}{\Gamma(s_T)(a_0s_T-\langle \boldsymbol{\alpha}|\mathbf{k} \rangle-1)}. \nonumber
\end{align}
\end{definition}

\begin{definition}
Let $\boldsymbol{\ell}=(\ell_{r+1},\ldots,\ell_{d-1}) \in \mathbb{N}_0^{d-r-1}$ be a tuple where $0 \leq r \leq d-1$. For all $\mathbf{s} \in \mathcal{H}_T(\delta,\delta+\frac{\varepsilon}{a_0})$ we set
\begin{align}
    J^{(r)}_{\boldsymbol{\ell}}(\mathbf{s}):=& \frac{1}{(2\pi \ii)^r}\int_{(\rho_1)} \cdots \int_{(\rho_r)} \frac{\Gamma(\widetilde{g}(s_T;\mathbf{z},-\boldsymbol{\ell}))\Gamma(g(s_T;\mathbf{z},-\boldsymbol{\ell}))}{\Gamma(s_T)} \label{eq:def_J_r} \\
    & \hspace{139pt} \times F(\mathbf{s};\mathbf{z},-\boldsymbol{\ell},g(s_T;\mathbf{z},-\boldsymbol{\ell})) \prod_{j=1}^r \Gamma(z_j) \, \dd z_j ,\nonumber \\
    J^{(0)}_{\boldsymbol{\ell}}(\mathbf{s}):=& \frac{F(\mathbf{s};-\boldsymbol{\ell},g(s_T;-\boldsymbol{\ell})) \Gamma(\widetilde{g}(s_T;-\boldsymbol{\ell}))\Gamma(g(s_T;-\boldsymbol{\ell}))}{\Gamma(s_T)} \nonumber
\end{align}
where we recall that $g$ and $\widetilde{g}$ are affine linear forms defined in Theorem \ref{th:analytic_continuation_zeta_general_case}.
\end{definition}

Observe that the integrals $J^{(r)}_{\boldsymbol{\ell}}$ and $I^{(r)}_{\mathbf{k}}$ are all well-defined thanks to the following useful lemma:
\begin{lemm}
\label{lemma:ineq_l_tilde_l_delta_epsilon}
Let $\boldsymbol{\ell}=(\ell_{r+1},\ldots,\ell_{d-1}) \in \mathbb{N}_0^{d-r-1}$ be a tuple, with $0 \leq r \leq d-1$. Then for all $\mathbf{s} \in \mathcal{H}_T(\delta,\delta+\frac{\varepsilon}{a_0})$ and all $x \in (-\infty,\rho_r]$, we have
\begin{align}
    &g(\sigma_T;\rho_1,\ldots,\rho_{r-1},x,-\boldsymbol{\ell})>0 \label{eq:ineq_l_k} \\
    &\widetilde{g}(\sigma_T;\rho_1,\ldots,\rho_{r-1},x,-\boldsymbol{\ell})>0 \label{eq:ineq_tilde_l_k} \\
    &a_0\sigma_T+\sum_{j=1}^{r-1} \alpha_j \rho_j+\alpha_r x-1 <0. \label{eq:ineq_delta_epsilon}
\end{align}
\end{lemm}

\begin{proof}
Observe that both functions
\begin{align*}
    &x \in \mathbb{R} \mapsto g(\sigma_T;\rho_1,\ldots,\rho_{r-1},x,-\boldsymbol{\ell}) \\
    &x \in \mathbb{R} \mapsto \widetilde{g}(\sigma_T;\rho_1,\ldots,\rho_{r-1},x,-\boldsymbol{\ell})
\end{align*}
are non-increasing, therefore it is enough to prove the two inequalities \eqref{eq:ineq_l_k} and \eqref{eq:ineq_tilde_l_k} when $x=\rho_r$. For a similar reason, it suffices to show that \eqref{eq:ineq_delta_epsilon} holds when $x=\rho_r$.

By \eqref{eq:ineq_delta_a0_delta_a1}, we have that $a_0 \delta<1-2\varepsilon$. By \eqref{eq:ineq_alpha_rho_epsilon} we have $\alpha_1\rho_1+\cdots +\alpha_r\rho_r \leq \alpha_1\rho_1+\cdots +\alpha_{d-1}\rho_{d-1}<\varepsilon$. Thus, we obtain
\begin{align*}
g(\sigma_T;\rho_1,\ldots,\rho_{r},-\boldsymbol{\ell}) &> 
\frac{-a_0(\delta+\frac{\varepsilon}{a_0})-\varepsilon+
\sum_{j=r+1}^d\alpha_j \ell_j+1}{\alpha_d}\\
&>
\frac{-a_0\delta-2\varepsilon+1}{\alpha_d}>0, 
\end{align*}
which proves \eqref{eq:ineq_l_k}.

By \eqref{eq:ineq_delta_a0_delta_a1}, we have $a_d \delta \geq a_1\delta>1+\varepsilon$. By \eqref{eq:ineq_alpha_rho_epsilon_bis}, we have $(\alpha_d-\alpha_1)\rho_1+\cdots +(\alpha_d-\alpha_r) \rho_{r} \leq (\alpha_d-\alpha_1)\rho_1+\cdots +(\alpha_d-\alpha_{d-1}) \rho_{d-1}<\varepsilon$. Thus, we obtain
$$ \widetilde{g}(\sigma_T;\rho_1,\ldots,\rho_{r},-\boldsymbol{\ell})
>\frac{a_d\delta-\varepsilon+\sum_{j=r+1}^d \alpha_j\ell_j
-1}{\alpha_d} 
>\frac{a_d\delta-\varepsilon-1}{\alpha_d}>0, $$
which proves \eqref{eq:ineq_tilde_l_k}.

By \eqref{eq:ineq_delta_a0_delta_a1}, we have $a_0\sigma_T<a_0\delta+\varepsilon<1-\varepsilon$. By \eqref{eq:ineq_alpha_rho_epsilon}, we have $\alpha_1\rho_1+\cdots +\alpha_r\rho_r<\varepsilon$. Thus we find
$$ a_0\sigma_T+\sum_{j=1}^{r} \alpha_j \rho_j-1 <0. $$
\end{proof}

\subsubsection{A first shift to the left}

\begin{prop}
\label{prop:formula_zeta_n_shift_left}
Let $M_d \in \mathbb{N}_0, \ \eta_d \in (0,1)$. For all $\mathbf{s} \in \mathcal{H}_{T}(\delta,\delta+\frac{\varepsilon}{a_0})$, we have
\begin{align*}
    \zeta_{n}(&\mathbf{s};\mathbf{P};\boldsymbol{\mu}_{n-1})=\frac{1}{\alpha_d} J^{(d-1)}_{\emptyset}(\mathbf{s}) + \sum_{k_d=0}^{M_d} \frac{(-1)^{k_d}}{k_d!} I^{(d-1)}_{k_d}(\mathbf{s}) \\
    & + \frac{1}{(2\pi \ii)^d\Gamma(s_T)} \int_{(\rho_1)} \cdots \int_{(\rho_{d-1})}\int_{(-M_{d}-\eta_{d})} \frac{F(\mathbf{s};\mathbf{z})\Gamma(s_T-|\mathbf{z}|) \prod_{j=1}^d \Gamma(z_j) \, \dd z_j}{a_0s_T+\langle \boldsymbol{\alpha}|\mathbf{z}\rangle -1}.
\end{align*}
Moreover, the last integral defines a holomorphic function on $\mathcal{H}_T(\varepsilon-M_d-\eta_d,\frac{1-\varepsilon+M_d+\eta_d}{a_0})$.
\end{prop}

\begin{proof}
First observe that by \eqref{eq:ineq_delta_sum_alpha_rho} and \eqref{eq:ineq_delta_sum_rho}, the set $\mathcal{H}_{T}(\delta,\delta+\frac{\varepsilon}{a_0})$ belongs to the domain $\mathcal{D}_{\boldsymbol{\rho}}(\mathbf{P})$, so we can use the integral expression \eqref{eq:integral_representation_zeta_n} of $\zeta_{n}(\mathbf{s};\mathbf{P};\boldsymbol{\mu}_{n-1})$. By shifting the path of integration from $(\rho_d)$ to $(-M_d-\eta_d)$ with respect to $z_d$ in \eqref{eq:integral_representation_zeta_n}, we encounter two different types of singularities:
\begin{align*}
    & z_d=-k_d && \text{ from } \Gamma(z_d) \qquad (0 \leq k_d \leq M_d), \\
    & z_d=\frac{1-a_0s_T-\alpha_1z_1-\cdots -\alpha_{d-1}z_{d-1}}{\alpha_d} && \text{ from } \frac{1}{a_0s_T+\sum_{j=1}^d \alpha_j z_j -1}.
\end{align*}
Note that by setting $\mathbf{z}=(z_1,\ldots,z_{d-1}) \in (\rho_1) \times \cdots \times (\rho_{d-1})$, this last singularity verifies $z_d=g(s_T;\mathbf{z})$. By Lemma \ref{lemma:ineq_l_tilde_l_delta_epsilon}, we have $\Re(g(s_T;\mathbf{z}))>0$. In particular, we see that the two types of singularities as described above do not overlap each other, meaning that $g(s_T;\mathbf{z}) \not \in \mathbb{Z}_{\leq 0}$ for all $\mathbf{z}=(z_1,\ldots,z_{d-1}) \in (\rho_1) \times \cdots \times (\rho_{d-1})$. By the residue theorem, we obtain
\begin{align*}
     &\zeta_{n}(\mathbf{s};\mathbf{P};\boldsymbol{\mu}_{n-1}) =\frac{1}{(2\pi \ii)^{d-1}\alpha_d}\int_{(\rho_{1})} \cdots \int_{(\rho_{d-1})} \frac{\Gamma(\widetilde{g}(s_T;\mathbf{z})) \Gamma(g(s_T;\mathbf{z}))}{\Gamma(s_T)} \\
    & \hspace{232pt} \times F(\mathbf{s};\mathbf{z},g(s_T;\mathbf{z})) \prod_{j=1}^{d-1} \Gamma(z_j) \, \dd z_{j} \\
    &+ \sum_{k_d=0}^{M_d}\frac{(-1)^{k_d}}{(2\pi \ii)^{d-1} k_d!} \int_{(\rho_1)} \cdots \int_{(\rho_{d-1})} \frac{\Gamma(s_T+k_d-|\mathbf{z}|)F(\mathbf{s};\mathbf{z},-k_d)\prod_{j=1}^{d-1} \Gamma(z_j) \, \dd z_{j}}{\Gamma(s_T) (a_0s_T\langle \boldsymbol{\alpha}|(\mathbf{z},-k_d) \rangle-1)} \\
    & + \frac{1}{(2\pi \ii)^d\Gamma(s_T)} \int_{(\rho_{1})} \cdots \int_{(\rho_{d-1})}\int_{(-M_d-\eta_d)} \frac{\Gamma(s_T-|\mathbf{z}|)F(\mathbf{s};\mathbf{z})\prod_{j=1}^d \Gamma(z_j) \, \dd z_j}{a_0s_T+\langle \boldsymbol{\alpha}|\mathbf{z} \rangle -1}.
\end{align*}
The first integral corresponds to $\frac{1}{\alpha_d} J^{(d-1)}_{\emptyset}(\mathbf{s})$. The integrals inside the sum correspond to $\frac{(-1)^{k_d}}{k_d!}I^{(d-1)}_{k_d}(\mathbf{s}) \ (0 \leq k_d \leq M_d)$. The last integral defines a holomorphic function on $\mathcal{H}_T(\varepsilon-M_d-\eta_d,\frac{1-\varepsilon+M_d+\eta_d}{a_0})$ because the function
$$ (s_T, \mathbf{z}) \mapsto \frac{1}{a_0s_T+\langle \boldsymbol{\alpha}|\mathbf{z} \rangle -1} $$
is holomorphic of moderate growth in the domain $(a_0\sigma_T+\sum_{j=1}^d \alpha_j \Re z_j<1)$, and the function 
$$(\mathbf{s},\mathbf{z}) \mapsto \Gamma\left(s_T-\sum_{j=1}^d z_j\right)$$ 
is also holomorphic of exponential decay in the domain $(\sigma_T>\sum_{j=1}^d \Re z_j)$.
\end{proof}

\subsubsection{Recursive formulas for the auxiliary integrals}

The two integrals $I^{(r)}_{\mathbf{k}}$ and $J^{(r)}_{\boldsymbol{\ell}}$ have some recursive structure, providing we assume some conditions on the complex number $s_T$, and the positive real numbers $\rho_1,\ldots,\rho_d$. For the two following propositions, we fix an integer $r \in [\![1,d-1]\!]$, a tuple $\mathbf{k}=(k_{r+1},\ldots,k_d) \in \mathbb{N}_0^{d-r}$ and a tuple $\boldsymbol{\ell}=(\ell_{r+1},\ldots,\ell_{d-1}) \in \mathbb{N}_0^{d-r-1}$.

\begin{prop}
\label{prop:iterative_structure_I_r}
Let $M_r \in \mathbb{N}_0$, $0<\eta_r<1$, and $\mathbf{s} \in \mathcal{H}_T(\delta,\delta+\frac{\varepsilon}{a_0})$. Then we have
\begin{align*}
    &I^{(r)}_{\mathbf{k}}(\mathbf{s})=\sum_{k_r=0}^{M_r}\frac{(-1)^{k_r}}{k_r!}I^{(r-1)}_{(k_r,\ldots,k_d)}(\mathbf{s}) \\
    &+\frac{1}{(2\pi \ii)^r\Gamma(s_T)} \int_{(\rho_{1})} \cdots \int_{(\rho_{r-1})} \int_{(-M_r-\eta_r)} \frac{\Gamma(s_T+|\mathbf{k}|-|\mathbf{z}|)F(\mathbf{s};\mathbf{z},-\mathbf{k}) \prod_{j=1}^r \Gamma(z_j) \, \dd z_j}{a_0s_T+\langle \boldsymbol{\alpha}|\mathbf{z},-\mathbf{k} \rangle -1}.
\end{align*}
Moreover, the last integral defines a holomorphic function on the domain
$$ \mathcal{H}_T\left(\varepsilon-M_r-\eta_r,\frac{1-\varepsilon+M_r+\eta_r}{a_0}\right). $$
\end{prop}

\begin{proof}
The proof of this recursive formula follows the same thread as the proof of Proposition \ref{prop:formula_zeta_n_shift_left}. We can perform a similar shift towards the left, but this time we pass through fewer singularities than in the previous case. Indeed, by \eqref{eq:ineq_delta_epsilon}, we obtain that the only singularities we encounter in the integral \eqref{eq:def_I_r} when shifting $z_r$ from $(\rho_r)$ to $(-M_r-\eta_r)$ come from $\Gamma(z_r)$. Via the residue theorem, we obtain
\begin{align*}
    I^{(r)}_{\mathbf{k}}(\mathbf{s}) =& \sum_{k_r=0}^{M_r}\frac{(-1)^{k_r}}{(2\pi \ii)^{r-1} k_r!} \int_{(\rho_1)} \cdots \int_{(\rho_{r-1})} \frac{\Gamma(s_T+k_r+|\mathbf{k}|-|\mathbf{z}|)}{\Gamma(s_T)} \\
    & \hspace{171pt} \times \frac{F(\mathbf{s};\mathbf{z},-k_r,-\mathbf{k})\prod_{j=1}^{r-1} \Gamma(z_j) \, \dd z_{j}}{a_0s_T+\langle \boldsymbol{\alpha}|\mathbf{z},-k_r,-\mathbf{k} \rangle-1} \\
    & + \frac{1}{(2\pi \ii)^r\Gamma(s_T)} \int_{(\rho_{1})} \cdots \int_{(\rho_{r-1})}\int_{(-M_r-\eta_r)} \frac{\Gamma(s_T+|\mathbf{k}|-|\mathbf{z}|)}{a_0s_T+\langle \boldsymbol{\alpha}|\mathbf{z},-\mathbf{k} \rangle -1} \\
    & \hspace{207pt} \times F(\mathbf{s};\mathbf{z},-\mathbf{k})\prod_{j=1}^r \Gamma(z_j) \, \dd z_j.
\end{align*}
Clearly, the last integral defines a holomorphic function on $\mathcal{H}_T(\varepsilon-M_r-\eta_r,\frac{1-\varepsilon+M_r+\eta_r}{a_0})$.
\end{proof}

\begin{prop}
\label{prop:iterative_structure_J_r}
Let $M_r \in \mathbb{N}_0$, $0<\eta_r<1$, and $\mathbf{s} \in \mathcal{H}_T(\delta,\delta+\frac{\varepsilon}{a_0})$. Then,
\begin{align*}
    J^{(r)}_{\boldsymbol{\ell}}(\mathbf{s})&=\sum_{\ell_r=0}^{M_r} \frac{(-1)^{\ell_r}}{\ell_r!} J^{(r-1)}_{(\ell_r,\ldots,\ell_d)}(\mathbf{s}) \\
    & + \frac{1}{(2\pi \ii)^r \Gamma(s_T)} \int_{(\rho_1)} \cdots \int_{(\rho_{r-1})}\int_{(-M_{r}-\eta_{r})} \Gamma(\widetilde{g}(s_T;\mathbf{z},-\boldsymbol{\ell}))\Gamma(g(s_T;\mathbf{z},-\boldsymbol{\ell})) \\
    & \hspace{155pt} \times F(\mathbf{s};\mathbf{z},-\boldsymbol{\ell},g(s_T;\mathbf{z},-\boldsymbol{\ell})) \prod_{j=1}^r \Gamma(z_j) \, \dd z_j.
\end{align*}
Moreover, the last integral defines a holomorphic function on the domain
$$ \mathcal{H}_T\left(\frac{1+\varepsilon-(a_d-a_r)(M_r+\eta_r)}{a_d},\frac{1-\varepsilon+\alpha_r(M_r+\eta_r)}{a_0}\right). $$
\end{prop}

\begin{proof}[Proof of Proposition \ref{prop:iterative_structure_J_r}]
Let $\mathbf{s} \in \mathcal{H}_T(\delta,\delta+\frac{\varepsilon}{a_0})$. By \eqref{eq:ineq_l_k} and \eqref{eq:ineq_tilde_l_k}, we know that there is no singularity coming from $\Gamma(\widetilde{g}(s_T;\mathbf{z},-\boldsymbol{\ell}))\Gamma(g(s_T;\mathbf{z},-\boldsymbol{\ell}))$ that we encounter when we shift the integration's path from $(\rho_r)$ to $(-M_r-\eta_r)$ with respect to $z_r$ in \eqref{eq:def_J_r}. We only encounter the singularities coming from $\Gamma(z_r)$. Therefore, by the residue theorem we obtain
\begin{align*}
    J^{(r)}_{\boldsymbol{\ell}}(\mathbf{s})=& \sum_{\ell_r=0}^{M_r} \frac{(-1)^{\ell_r}}{(2\pi \ii)^{r-1}\ell_r!} \int_{(\rho_1)} \cdots \int_{(\rho_{r-1})}  \frac{\Gamma(\widetilde{g}(s_T;\mathbf{z},-\ell_r,-\boldsymbol{\ell}))\Gamma(g(s_T;\mathbf{z},-\ell_r,-\boldsymbol{\ell}))}{\Gamma(s_T)} \\
    & \hspace{100pt} \times F(\mathbf{s};\mathbf{z},-\ell_{r},-\boldsymbol{\ell},g(s_T;\mathbf{z},-\ell_r,-\boldsymbol{\ell}))\prod_{j=1}^{r-1} \Gamma(z_j) \dd z_j \\
    & + \frac{1}{(2\pi \ii)^r\Gamma(s_T)} \int_{(\rho_1)} \cdots \int_{(\rho_{r-1})}\int_{(-M_{r}-\eta_{r})} \Gamma(\widetilde{g}(s_T;\mathbf{z},-\boldsymbol{\ell}))\Gamma(g(s_T;\mathbf{z},-\boldsymbol{\ell}))\\
    & \hspace{147pt} \times F(\mathbf{s};\mathbf{z},-\boldsymbol{\ell},g(s_T;\mathbf{z},-\boldsymbol{\ell})) \prod_{j=1}^r \Gamma(z_j) \, \dd z_j.
\end{align*}
The summand corresponds to $J^{(r-1)}_{(\ell_r,\boldsymbol{\ell})}(\mathbf{s})$. The last integral is holomorphic over $\mathcal{H}_T(\frac{1+\varepsilon-(\alpha_d-\alpha_r)(M_r+\eta_r)}{a_d},\frac{1-\varepsilon+\alpha_r(M_r+\eta_r)}{a_0})$ because on this set, both affine linear forms $\widetilde{g}(s_T;\mathbf{z},-\boldsymbol{\ell})$ and $g(s_T;\mathbf{z},-\boldsymbol{\ell})$ have a positive real part for all $\mathbf{z}=(z_1,\ldots,z_r) \in (\rho_1) \times \cdots \times (\rho_{r-1}) \times (-M_r-\eta_r)$.
\end{proof}

\subsection{Meromorphic continuation}
\label{subsection:meromorphic_continuation_zeta_n}

In order to obtain a meromorphic continuation of $\zeta_{n}(\mathbf{s};\mathbf{P};\boldsymbol{\mu}_{n-1})$, we will first apply Proposition \ref{prop:formula_zeta_n_shift_left}, and then we apply Proposition \ref{prop:iterative_structure_I_r} and Proposition \ref{prop:iterative_structure_J_r} recursively.

\begin{proof}[Proof of Theorem \ref{th:analytic_continuation_zeta_general_case}]
Note that by \eqref{eq:ineq_delta_a0_delta_a1}, we have that 
$$ \frac{1-\varepsilon+y(\mathbf{M},\boldsymbol{\eta})}{a_0}> \frac{1-\varepsilon}{a_0}>\delta.$$
Therefore the set 
$$ \mathcal{H}':=\mathcal{H}_{\varepsilon}(\mathbf{M},\boldsymbol{\eta}) \cap \mathcal{H}_T(\delta,\delta+\frac{\varepsilon}{a_0})$$
is a non-empty open connected subset of $\mathbb{C}^T$. For all $\mathbf{s} \in \mathcal{H}'$, we can first apply Proposition \ref{prop:formula_zeta_n_shift_left} in order to express $\zeta_{n}(\mathbf{s};\mathbf{P};\boldsymbol{\mu}_{n-1})$ as a finite sum of integrals of the form $I^{(d-1)}$ and $J^{(d-1)}$, plus a holomorphic function. Since $\mathbf{s} \in \mathcal{H}' \subset \mathcal{H}_T(\delta,\delta+\frac{\varepsilon}{a_0})$, we can iteratively apply Proposition \ref{prop:iterative_structure_I_r} and Proposition \ref{prop:iterative_structure_J_r} to simplify both integrals $I^{(d-1)}$ and $J^{(d-1)}$.

By the induction principle, we prove that for all $\mathbf{s} \in \mathcal{H}'$ and all $0 \leq r \leq d-1$, we have
\begin{align}\label{tuika}
    \zeta_{n}(\mathbf{s};\mathbf{P};\boldsymbol{\mu}_{n-1})=&\frac{H(\mathbf{s})}{\Gamma(s_T)}+ \sum_{k_{r+1}=0}^{M_{r+1}} \cdots \sum_{k_d=0}^{M_d} \frac{(-1)^{|\mathbf{k}|}}{k_{r+1}!\cdots k_d!} I^{(r)}_{\mathbf{k}}(\mathbf{s}) \notag\\
    &+ \frac{1}{\alpha_d}\sum_{\ell_{r+1}=0}^{M_{r+1}} \cdots \sum_{\ell_{d-1}=0}^{M_{d-1}} \frac{(-1)^{|\boldsymbol{\ell}|}}{\ell_{r+1}!\cdots \ell_{d-1}!} J^{(r)}_{\boldsymbol{\ell}}(\mathbf{s}),
\end{align}
where $H(\mathbf{s})$ is a holomorphic function on $\mathcal{H}_{\varepsilon}(\mathbf{M},\boldsymbol{\eta})$. The base case follows from Proposition \ref{prop:formula_zeta_n_shift_left}. The induction step follows from both Proposition \ref{prop:iterative_structure_I_r} and Proposition \ref{prop:iterative_structure_J_r}. Observe that all three propositions involve three functions defined via integrals, and they are all holomorphic on $\mathcal{H}_{\varepsilon}(\mathbf{M},\boldsymbol{\eta})$. 
The case $r=0$ of the above \eqref{tuika} gives
\eqref{eq:analytic_continuation_zeta_general_case} 
in Theorem \ref{th:analytic_continuation_zeta_general_case}.
 Except for $H(\mathbf{s})$, all functions appearing on the
 right-hand side of 
\eqref{eq:analytic_continuation_zeta_general_case} are
meromorphic 
on the whole complex space, with respect to $\mathbf{s}$.
\end{proof}

\begin{rem}
Working in the set $\mathbf{s} \in \mathcal{H}(\delta,\delta+\frac{\varepsilon}{a_0})$, with $\delta, \varepsilon$ as in Assumption \ref{assumption:delta_epsilon_rho} simplifies the expressions obtained via the residue theorem for $I^{(r)}_{\mathbf{k}}$ and $J^{(r)}_{\boldsymbol{\ell}}$. Indeed, by taking such $\mathbf{s}$, we avoid dealing with singularities and residues coming from the product $\Gamma(\widetilde{g}(s_T;\mathbf{z},-\boldsymbol{\ell}))\Gamma(g(s_T;\mathbf{z},-\boldsymbol{\ell}))$ in the integrand of $J^{(r)}_{\boldsymbol{\ell}}$ ($1 \leq r \leq d-1$), and we also avoid the singularities coming from $\frac{1}{a_0s_T+\sum_{j=1}^d \alpha_j z_j -1}$ in the integrand of $I^{(r)}_{\mathbf{k}}$ ($1 \leq r \leq d-1$).
\end{rem}

\subsection{Special values at non-positive integers}
\label{subsection:values_zeta_n_general_case}

\begin{proof}[Proof of Theorem \ref{th:value_zeta_n}]
We apply Theorem \ref{th:analytic_continuation_zeta_general_case} with a tuple $\mathbf{M}=(M_1,\ldots,M_d)\in \mathbb{N}_0^d$ sufficiently large and $\boldsymbol{\eta}=(\eta_1,\ldots,\eta_d) \in (0,1)^d$ such that $-\mathbf{N} \in \mathcal{H}_{\varepsilon}(\mathbf{M},\boldsymbol{\eta})$. We then obtain
\begin{align*}
    \zeta_{n}(\mathbf{s};\mathbf{P};\boldsymbol{\mu}_{n-1})=& -\sum_{k_1=0}^{M_1} \cdots \sum_{k_d=0}^{M_d}  \frac{(-1)^{|\mathbf{k}|}F(-\mathbf{N};-\mathbf{k})}{\mathbf{k}!(a_0N_T+\langle \boldsymbol{\alpha}|\mathbf{k}\rangle+1)}\lim_{s_T \to -N_T}\frac{\Gamma(s_T+|\mathbf{k}|)}{\Gamma(s_T)} \\ 
    & +\sum_{\ell_1=0}^{M_1} \cdots \sum_{\ell_{d-1}=0}^{M_{d-1}} \frac{(-1)^{|\boldsymbol{\ell}|}F(-\mathbf{N};-\boldsymbol{\ell},g(-N_T;-\boldsymbol{\ell}))}{\alpha_d\boldsymbol{\ell}!} \\
    & \hspace{119pt} \times \lim_{s_T \to -N_T}\frac{\Gamma(\widetilde{g}(s_T;-\boldsymbol{\ell}))\Gamma(g(s_T;-\boldsymbol{\ell}))}{\Gamma(s_T)}.
\end{align*} 
It remains to compute the two limits in the formula above.

Let $\mathbf{k}=(k_1,\ldots,k_d) \in \mathbb{N}_0^d$, by the functional relation for $\Gamma(s)$ we get
\[
    \lim_{s_T \to -N_T} \frac{\Gamma(s_T+|\mathbf{k}|)}{\Gamma(s_T)} =
    \begin{cases}
        \frac{(-1)^{|\mathbf{k}|}N_T!}{(N_T-|\mathbf{k}|)!} & \text{ if } |\mathbf{k}| \leq N_T, \\
        0 &\text{ otherwise.}
    \end{cases}
\]

Let $\boldsymbol{\ell}=(\ell_1,\ldots,\ell_{d-1}) \in \mathbb{N}_0^{d-1}$. For the evaluation of the remaining function, we first notice that $\widetilde{g}(-N_T;-\boldsymbol{\ell}) \geq -\frac{a_dN_T+1}{\alpha_d}$. For all $i \in [\![ 0,\lfloor \frac{a_dN_T+1}{\alpha_d} \rfloor ]\!]$, we see that
$\widetilde{g}(-N_T;-\boldsymbol{\ell})=-i$ is equivalent to
\begin{align*}
\sum_{j=1}^{d} (\alpha_d-\alpha_j) \ell_j=1+a_dN_T-i\alpha_d,
\end{align*}
which is further equivalent to
\begin{align*}
    g(-N_T;-\boldsymbol{\ell})=|\boldsymbol{\ell}|+i-N_T.
\end{align*}
It is clear by definition of $g$ that $g(-N_T;-\boldsymbol{\ell}) \in \frac{1}{\alpha_d}\mathbb{N}$, so it is a positive number. Finally, again by the functional relation for $\Gamma(s)$ we get
\[
    \lim_{s_T \to -N_T} \frac{\Gamma(\widetilde{g}(s_T;-\boldsymbol{\ell}))}{\Gamma(s_T)} =
    \begin{cases}
        \frac{(-1)^{N_T}\alpha_dN_T!}{(-1)^ia_di!} & \text{if there exists } 0 \leq i \leq \left\lfloor \frac{a_dN_T+1}{\alpha_d}\right\rfloor \text{ s.t.} \\
        & \sum_{j=1}^{d-1} (a_d-a_j) \ell_j=1+a_dN_T-i\alpha_d, \\
        0 &\text{otherwise.}
    \end{cases}
\]
Therefore, we obtain
\begin{align*}
    \zeta_{n}(&-\mathbf{N};\mathbf{P};\boldsymbol{\mu}_{n-1})=-\sum_{\substack{k_1,\ldots,k_d \geq 0 \\ |\mathbf{k}| \leq N_T}} \binom{N_T}{(k_1,\ldots,k_d,N_T-|\mathbf{k}|)} \frac{F(-\mathbf{N};-\mathbf{k})}{a_0N_T+\langle \boldsymbol{\alpha}|\mathbf{k}\rangle+1} \\
    &+ \sum_{i=0}^{\left\lfloor \frac{a_dN_T+1}{\alpha_d}\right\rfloor} \sum_{\substack{\ell_1,\ldots,\ell_{d-1} \geq 0 \\ \sum_{j=1}^{d-1} (\alpha_d-\alpha_j) \ell_j=1+a_dN_T-i\alpha_d}} \frac{(-1)^{|\boldsymbol{\ell}|+i+N_T}\alpha_dN_T!}{\alpha_da_di!\boldsymbol{\ell}!}(|\boldsymbol{\ell}|+i-N_T-1)! \\
    & \hspace{231pt} \times F(-\mathbf{N};-\boldsymbol{\ell},|\boldsymbol{\ell}|+i-N_T).
\end{align*}
It remains to write explicitly the special values of $F$ appearing in the previous formula. For $\mathbf{k}=(k_1,\ldots,k_d) \in \mathbb{N}_0^d$ such that $|\mathbf{k}|\leq N_T$, by using the relation $\zeta(-m)=-\frac{\widetilde{B}_{m+1}}{m+1}$ we obtain
\begin{align*}
    F(-\mathbf{N};-\mathbf{k})=&\widetilde{B}_{a_0N_T+\langle \boldsymbol{\alpha}|\mathbf{k}\rangle+1} \\
    & \hspace{38pt} \times \zeta_{n-1}^{\dc}(-N_1,\ldots,-N_{T-1},-N_T+|\mathbf{k}|,-\mathbf{k};\mathbf{P(Q)};\boldsymbol{\mu}_{n-1}).
\end{align*}
For $\boldsymbol{\ell}=(\ell_1,\ldots,\ell_{d-1}) \in \mathbb{N}_0^{d-1}$ such that $\alpha_d|\boldsymbol{\ell}|-\langle \boldsymbol{\alpha},\boldsymbol{\ell}\rangle=1+a_dN_T-i\alpha_d$, we have
\begin{align*}
    F(-\mathbf{N};-\boldsymbol{\ell},|\boldsymbol{\ell}&|+i-N_T) \\
    &= \zeta_{n-1}^{\dc}\left((-N_1,\ldots,-N_{T-1},-i,-\boldsymbol{\ell},|\boldsymbol{\ell}|+i-N_T);\mathbf{P(Q)};\boldsymbol{\mu}_{n-1}\right).
\end{align*}
\end{proof}

\section{Examples}
\label{application}

In Theorem \ref{th:value_zeta_n}, one can observe that the formula obtained for the special value $\zeta_{n}(-\mathbf{N};\mathbf{P};\boldsymbol{\mu}_{n-1})$ involves values of fully twisted multiple zeta-functions at integer points $\mathbf{s} \in \mathbb{Z}_{\leq 0}^{T+d-1} \times \mathbb{N}_0$, and these values are rather mysterious. In the following three examples, we consider specific polynomials $P_1,\ldots,P_T$ such that this positive value phenomenon can be taken care of. In the first example, we assume that $P_T$ is a polynomial with a constant term in front of the highest degree monomial of $X_n$. 
In the second example, we study the simple case where $n=T=1$.
Lastly, in the third example, we look at a specific case where the number of variables in the Dirichlet series is $n=2$. In this case, we find that some special values $\zeta_{2}(-\mathbf{N};\mathbf{P};\boldsymbol{\mu}_1)$ are transcendental numbers. 

\subsection{\texorpdfstring{1\textsuperscript{st}}{First} Example: \texorpdfstring{$P_T$}{P} with a constant leading coefficient}

Let $\mathbb{K}$ be a subfield of $\mathbb{R}$. Let $\boldsymbol{\mu}_{n-1}=(\mu_1,\ldots,\mu_{n-1}) \in (\mathbb{T} \setminus \{ 1 \})^{n-1}$. Let $\mathbf{P}=(P_1,\ldots,P_T)$ be polynomials with coefficients in $\mathbb{K}$ such that $P_j \in \mathbb{K}[X_1,\ldots,X_{n-1}]$ $(1 \leq j \leq T-1)$, and are polynomials satisfying HDF. Assume that $P_T$ has a constant leading coefficient in front of the monomial $X_n^{a_d}$, i.e.
$$ P_T(X_1,\ldots,X_n) = \sum_{j=0}^{d-1} Q_j(X_1,\ldots,X_{n-1}) X_n^{a_j}+cX_n^{a_d}, $$
with $Q_0,\ldots,Q_{d-1}$ nonzero polynomials satisfying HDF and verifying \eqref{eq:condition_polynomials_tends_infinity}, and $c \in \mathbb{K} \setminus \{0\}$. 
It is easy to see by analytic continuation that in this case 
\begin{align*}
    \zeta_{n-1}^{\dc}(&s_1,\ldots,s_{T},z_1,\ldots, z_d;\mathbf{P(Q)};\boldsymbol{\mu}_{n-1})\\
    &= c^{-z_d} \zeta_{n-1}^{\dc}(s_1,\ldots,s_{T},z_1,\ldots, z_{d-1};(P_1,\ldots,P_{T-1}, Q_0,\ldots, Q_{d-1});\boldsymbol{\mu}_{n-1}).
\end{align*}
So, by applying Theorem \ref{th:value_zeta_n} to $\zeta_n(\mathbf{s};\mathbf{P};\mu)$ and using de Crisenoy's result (namely, Proposition \ref{thm_dC}), it follows that
$$ \zeta_{n}(-\mathbf{N};\mathbf{P};\boldsymbol{\mu}_{n-1}) \in \mathbb{K}(\boldsymbol{\mu}_{n-1}). $$
We can actually make this result fully explicit in terms of Bernoulli numbers and Stirling numbers of the second kind, but the formula is lengthy.

\subsection{\texorpdfstring{2\textsuperscript{nd}}{Second} Example: \texorpdfstring{$n=T=1$}{n and T equal 1}}

We now study the case $n=T=1$. We set 
$$ P(X)=c_0X^{a_0}+\cdots +c_dX^{a_d},$$ 
where $a_0<\cdots <a_d$, and $c_0,\ldots,c_d > 0$, and we assume $d \neq 0$. Then, we can apply Theorem \ref{th:value_zeta_n} to get
\begin{align*}
    &\zeta_1(-N;P;1)=-c_0^{N}\sum_{\substack{k_1,\ldots,k_d \geq 0 \\ |\mathbf{k}| \leq N}} \binom{N}{(\mathbf{k},N-|\mathbf{k}|)} \frac{\widetilde{B}_{a_0N+\langle \boldsymbol{\alpha},\mathbf{k}\rangle+1}}{a_0N+\langle \boldsymbol{\alpha},\mathbf{k} \rangle+1}\prod_{j=1}^d \left(\frac{c_j}{c_0} \right)^{k_j} \\
    &+\sum_{i=0}^{\left\lfloor \frac{a_dN+1}{\alpha_d}\right\rfloor} \sum_{\substack{\ell_1,\ldots,\ell_{d-1} \geq 0 \\ \sum_{j=1}^{d-1} (a_d-a_j) \ell_j \\ \qquad =1+a_dN-i\alpha_d}} \frac{(-1)^{|\boldsymbol{\ell}|+i+N}(|\boldsymbol{\ell}|+i-N-1)!N!}{a_di!\boldsymbol{\ell}!} c_0^{i} c_d^{N-i} \prod_{j=1}^{d-1} \left(\frac{c_j}{c_d} \right)^{\ell_j}.
\end{align*}
It should be mentioned that a similar result was proved in Salinas' thesis \cite{Salinas}, via different techniques.


\subsection{\texorpdfstring{3\textsuperscript{rd}}{Third} Example: \texorpdfstring{$n=T=2$}{n and T equal 2}}

Let $n=2$, $\mu \in \mathbb{T} \setminus \{ 1 \}$. We consider the polynomials 
$$ P_1(X_1):=X_1^r, \qquad P_2(X_1,X_2):=\sum_{j=0}^d c_{j} X_1^{b_j} X_2^{a_j} $$
with $a_0<\cdots <a_d$. With those polynomials, for all $s \in \mathbb{C}$, $\mathbf{z}\in \mathbb{C}^{d+1}$ we have
$$ \zeta_1^{\dc}(s,\mathbf{z};P_1,c_0X_1^{b_0},\ldots,c_dX_1^{b_d};\mu)= \left( \prod_{j=0}^d c_j^{-z_j} \right) \zeta_{\mu}(rs+\sum_{j=0}^d b_j z_j). $$
Therefore
by Theorem \ref{th:value_zeta_n} we obtain for $-\mathbf{N}=(-N_1,-N_2) \in \mathbb{Z}_{\leq 0}^2$,
\begin{align}
    \zeta_{2}(-&\mathbf{N};\mathbf{P};\mu)=-\sum_{\substack{k_1,\ldots,k_d \geq 0 \\ |\mathbf{k}| \leq N_2}} \binom{N_2}{(k_1,\ldots,k_d,N_2-|\mathbf{k}|)} \frac{\widetilde{B}_{a_0N_2+\langle \boldsymbol{\alpha}|\mathbf{k}\rangle+1}}{a_0N_2+\langle \boldsymbol{\alpha}|\mathbf{k} \rangle+1} \label{eq:values_special_case_transcendental} \\
    & \hspace{52pt} \times c_0^{N_2} \left( \prod_{j=1}^d \frac{c_j^{k_j}}{c_0^{k_j}} \right) \zeta_{\mu}\left(-rN_1+b_0(|\mathbf{k}|-N_2)-b_1k_1-\cdots-b_d k_d\right) \nonumber \\
    &\hspace{3pt}+ \frac{1}{a_d}\sum_{i=0}^{\left\lfloor \frac{a_dN_2+1}{a_d-a_0}\right\rfloor} \sum_{\substack{\ell_1,\ldots,\ell_{d-1} \geq 0 \\ \sum_{j=1}^{d-1} (a_d-a_j) \ell_j=1+a_dN_2-i\alpha_d}} \frac{(-1)^{|\boldsymbol{\ell}|+i+N_2}N_2!(|\boldsymbol{\ell}|+i-N_2-1)!}{\boldsymbol{\ell}!} \nonumber \\
    & \hspace{68pt} \times \frac{c_0^i \prod_{j=1}^{d-1} c_j^{\ell_j}}{c_d^{|\boldsymbol{\ell}|+i-N_2}} \zeta_{\mu}(-rN_1-b_0i-\sum_{j=1}^{d-1}b_j\ell_j+b_d(|\boldsymbol{\ell}|+i-N_2)). \nonumber
\end{align}

\begin{example}\label{transcendental}
Let us consider the special case
$$P_1 (X_1)=1 \quad \text{and} \quad P_2(X_1, X_2)=1+ X_2+ X_1^q X_2^2. $$
Let $a_1=1, a_2=2, \alpha_1=1, \alpha_2=2$. By applying \eqref{eq:values_special_case_transcendental}, and by performing a change of variable, we obtain
\begin{align*}
    \zeta_{2}(-\mathbf{N};\mathbf{P};\mu)=&-\sum_{\substack{k_1,k_2 \geq 0 \\ k_1+k_2 \leq N_2}} \binom{N_2}{(k_1,k_2,N_2-k_1-k_2)} \frac{\widetilde{B}_{\langle \boldsymbol{\alpha}|\mathbf{k}\rangle+1}}{\langle \boldsymbol{\alpha}|\mathbf{k} \rangle+1} \zeta_{\mu}(-qk_2)\\
    &+ \sum_{i=0}^{N_2} \frac{(-1)^{1+i+N_2}N_2!(N_2-i)!}{2(2N_2-2i+1)!i!} \zeta_{\mu}(q(N_2-i+1), \\
    =&-\sum_{\substack{k_1,k_2 \geq 0 \\ k_1+k_2 \leq N_2}} \binom{N_2}{(k_1,k_2,N_2-k_1-k_2)} \frac{\widetilde{B}_{\langle \boldsymbol{\alpha}|\mathbf{k}\rangle+1}}{\langle \boldsymbol{\alpha}|\mathbf{k} \rangle+1} \zeta_{\mu}(-qk_2)\\
    &- \sum_{i=0}^{N_2} \frac{(-1)^{i}N_2!i!}{2(2i+1)!(N_2-i)!} \zeta_{\mu}(q(i+1)),
\end{align*}
These last values can be transcendental, depending on the torsion $\mu$, and the parity of $q$. We now assume that $q$ is a positive even integer, and $\mu=-1$. We know that $\zeta_{-1}(s)=(2^{1-s}-1)\zeta(s)$. Observe that the first sum is a rational number. Moreover, at positive even integers, the zeta values correspond to rational multiples of positive powers of $\pi$. Therefore, $\zeta_{2}(-\mathbf{N};\mathbf{P};-1)$ corresponds to a non-constant rational polynomial in $\pi$. It follows that $\zeta_{2}(-\mathbf{N};\mathbf{P};-1)$ is a transcendental number.
\end{example}


\par
D. Essouabri: Université Jean Monnet, Institut Camille Jordan (CNRS, UMR 5208), 20 rue du Docteur Rémy Annino, Saint-Étienne 42000, France.\\
E-mail address: driss.essouabri@univ-st-etienne.fr
\par
K. Matsumoto: Graduate School of Mathematics, Nagoya University,
Chikusa-ku, Nagoya 464-8602, Japan; and Center for General Education, 
Aichi Institute of Technology, Yakusa-cho, Toyota 470-0392, Japan\\
E-mail address: kohjimat@math.nagoya-u.ac.jp
\par
S. Rutard: Graduate School of Mathematics, Nagoya University,
Chikusa-ku, Nagoya 464-8602, Japan\\
E-mail address: simon.rutard@math.nagoya-u.ac.jp

\begin{thebibliography}{999}

\bibitem{AET01}S. Akiyama, S. Egami and Y. Tanigawa,
Analytic continuation of multiple zeta-functions and their values at non-positive integers,
Acta Arith. {\bf 98} (2001), 107--116.
\bibitem{ChMa16}Y. Choie and K. Matsumoto,
Functional equations for double series of Euler type with
coefficients,
Adv. Math. {\bf 292} (2016), 529--557.
\bibitem{ChMa17}Y. Choie and K. Matsumoto,
Functional equations for double series of Euler-Hurwitz-Barnes type with coefficients, 
in RIMS K{\^o}ky{\^u}roku Bessatsu {\bf B68} (2017), 91--109.
\bibitem{dC}M. de Crisenoy,
Values at $T$-tuples of negative integers of twisted multivariable zeta series associated to polynomials of several variables, Compositio Math. {\bf 142} (2006), no. 6, 1373--1402.
\bibitem{dCE08}M. de Crisenoy and D. Essouabri,
Relations between values at {$T$}-tuples of negative integers of twisted multivariable zeta series associated to polynomials of several variables, J. Math. Soc. Japan {\bf 60} (2008), 1--16.
\bibitem{Esso97}D. Essouabri,
Singularit{\'e}s de s{\'e}ries de Dirichlet associ{\'e}es {\`a} des polyn{\^o}mes de plusieurs
variables et application en th{\'e}orie analytique des nombres, 
Ann. Inst. Fourier {\bf 47} (1997), 429--483.
\bibitem{Esso05}D. Essouabri,
Zeta functions associated to Pascal's triangle mod p, Japanese J. Math. {\bf 31} (2005), 157--174.
\bibitem{Esso12}D. Essouabri,
Height zeta functions on generalized projective toric varieties, in ``Recent Trends on Zeta functions in Algebra and Geometry", Contemp. Math., Amer. Math. Soc. and Real Soc. 
Mat. Espanol., 2012, 34pp.
\bibitem{CMUSP}D. Essouabri and K. Matsumoto,
Values at non-positive integers of partially twisted multiple zeta-functions I, Comment. Math. Univ. St. Pauli {\bf 67} (2019), no. 1, 83--100.
\bibitem{EM20}D. Essouabri and K. Matsumoto,
Values at non-positive integers of generalized Euler-Zagier multiple zeta-functions,
Acta Arith. {\bf 193} (2020), 109--131. 
\bibitem{EM21}D. Essouabri and K. Matsumoto,
Values of multiple zeta functions with polynomial denominators at non-positive integers,
Intern. J. Math. {\bf 32} (2021), 2150038, 41pp. 
\bibitem{Komo10}Y. Komori,
An integral representation of multiple Hurwitz-Lerch zeta functions and generalized multiple
Bernoulli numbers, Quart. J. Math. (Oxford) {\bf 61} (2010), 437--496. 
\bibitem{KMT11}Y. Komori, K. Matsumoto and H. Tsumura,
Functional equations for double $L$-functions and values at
non-positive integers,
Intern. J. Number Theory {\bf 7} (2011), 1441--1461.
\bibitem{Mats04}K. Matsumoto,
Functional equations for double zeta-functions,
Math. Proc. Cambridge Phil. Soc. {\bf 136} (2004), 1--7.
\bibitem{MOW20}K. Matsumoto, T. Onozuka and I. Wakabayashi,
Laurent series expansions of multiple zeta-functions of Euler-Zagier type at integer points,
Math. Z. {\bf 295} (2020), 623--642.
\bibitem{Mell}H. Mellin,
Eine Formel f{\"u}r den Logarithmus transcendenter Funktionen von endlichem Geschlecht,
Acta Soc. Sci. Fenn. {\bf 29}, no. 4 (1900), 3--49.
\bibitem{MuOn}H. Murahara and T. Onozuka,
Asymptotic behavior of the Hurwitz-Lerch multiple zeta functions
at non-positive integer points,
Acta Arith. {\bf 205} (2022), 191--210.
\bibitem{Onoz13}T. Onozuka,
Analytic continuation of multiple zeta-functions and the asymptotic behavior at non-positive
integers, Funct. Approx. Comment. Math. {\bf 49} (2013), 331--348.
\bibitem{Rutard}S. Rutard,
Values and derivative values at nonpositive integers of generalized multiple Hurwitz zeta functions, to appear in J. Math. Soc. Japan, arXiv:2312.04725.
\bibitem{Saha22}B. Saha,
Multiple Stieltjes constants and Laurent type expansion of the multiple zeta functions at
integer points, Selecta Math. {\bf 28} (2022), paper no. 6, 41pp.
\bibitem{Salinas}C. Salinas Zavala, Multivariable zeta functions, multivariable tauberian theorems and applications, PhD thesis (2021), Universit{\'e} de Lyon; Universit{\'e} nationale d’ing{\'e}nierie (Lima).
\bibitem{Sasa09}Y. Sasaki,
Multiple zeta values for coordinatewise limits at non-positive
integers,
Acta Arith. {\bf 136} (2009), 299--317.
\bibitem{WW}E. T. Whittaker and G. N. Watson,
A Course of Modern Analysis, 4th ed.,
Cambridge Univ. Press, Cambridge, 1927.
\end{thebibliography}
\end{document}